\documentclass[a4paper,12pt]{article}
\usepackage{amssymb,amsmath,amsfonts,amsthm,a4wide}
\usepackage{graphicx}
\numberwithin{equation}{section}

\providecommand{\abs}[1]{\left\vert#1\right\vert}

\providecommand{\pnorm}[2]{\left\Vert#1\right\Vert_{L^{#2}}}
\providecommand{\pnormspace}[3]{\left\Vert#1\right\Vert_{L^{#2}(#3)}}
\providecommand{\pqnorm}[3]{\left\Vert#1\right\Vert_{L^{#2,#3}}}
\providecommand{\pqnormspace}[4]{\left\Vert#1\right\Vert_{L^{#2,#3}(#4)}}

\providecommand{\Rn}[1]{\mathbb{R}^{#1}}

\providecommand{\wnorm}[1]{\lvert \mspace{-1.8mu} \lvert \mspace{-1.8mu} \lvert #1 \rvert \mspace{-1.8mu} \rvert \mspace{-1.8mu} \rvert_{L^{2,\infty}}}

\def\({\left(}
\def\){\right)}
\def\l|{\left|}
\def\r|{\right|}
\def\ep{\varepsilon}
\def\mr{\mathbb{R}}
\def\mz{\mathbb{Z}}

\def\si{\mathbb{S}^1}
\def\p{\partial}
\def\lep{\abs{\mathrm{log }\ \ep}}

\def\nab{\nabla}

\def\om{\Omega}
\def\io{\int_{\Omega}}
\def\bo{\partial \Omega}
\def\hal{\frac{1}{2}}
\def\ro{\rho}

\def\lti{L^{2,\infty}}
\def\dist{\mathrm{dist}}

\DeclareMathOperator{\curl}{curl}

\newtheorem{lem}{Lemma}[section]
\newtheorem{cor}[lem]{Corollary}
\newtheorem{prop}[lem]{Proposition}
\newtheorem{thm}{Theorem}

\title{Lorentz Space Estimates for the Ginzburg-Landau Energy}
\author{Sylvia Serfaty\footnote{Supported by NSF CAREER grant \# DMS0239121 and a
Sloan Foundation Fellowship }  and Ian Tice\footnote{Supported by
an NSF Graduate Research Fellowship}\\
{\small Courant Institute of Mathematical Sciences}\\
{\small 251 Mercer St., New York, NY 10012}\\
{\small\tt serfaty@cims.nyu.edu, tice@cims.nyu.edu}}

\begin{document}

\maketitle
\begin{abstract}
In this paper we prove novel lower bounds for the Ginzburg-Landau
energy with or without magnetic field. These bounds rely on an
improvement of  the ``vortex balls construction'' estimates by
extracting a new positive term in the energy lower bounds.
 This extra term
can be conveniently estimated through a Lorentz space norm, on which it thus
 provides an upper bound.  The Lorentz space $L^{2,\infty}$ we use is
critical with respect to the expected vortex profiles and can
serve to estimate the total number of vortices and get improved
convergence results.
\end{abstract}


\section{Introduction}

\subsection{Motivation}In this paper we consider the  Ginzburg-Landau ``free
energy"
\begin{equation}
\label{glfree} F_\ep(u,A)= \hal \io\abs{\nab_A u}^2 + \abs{\curl A}^2 +
\frac{(1-|u|^2)^2}{2\ep^2}.
\end{equation}
Here $\Omega$ is a bounded regular  two dimensional domain  of
$\mr^2$, $u $ is a complex-valued function, and $A\in \mr^2$ is a
vector field in $\om$. This functional is the free energy of the model of
superconductivity developed by Ginzburg and Landau.
 In the model, $A$ is the vector-potential of
the magnetic field, the function $h:=\curl A=\p_1 A_2 - \p_2 A_1$
is the induced magnetic field, and the complex-valued function $u$
is the ``order parameter" indicating the local state of the
material (normal or superconducting): $|u|^2$ is the local density
of superconducting electrons. The notation $\nab_A $ refers to the
covariant gradient, which acts according to $\nab_A u = (\nab - i A) u$.

We are interested in the regime of small $\ep$: $\ep$ corresponds
to a material constant, and small $\ep$ implies type-II
superconductivity.  In this regime, $u$ (because it  is complex-valued)
can have zeroes  with a nonzero topological degree.  These defects
are called the {\it vortices} of $u$ and are the crucial
objects of interest.

By setting $A \equiv 0$ we are led to studying the simpler
Ginzburg-Landau energy ``without magnetic field":
\begin{equation}
\label{eep} E_\ep(u)= \hal \io \abs{\nab u}^2 +
\frac{(1-|u|^2)^2}{2\ep^2}.
\end{equation}
All our results will thus apply to this energy as well, by setting
$A\equiv 0$.

These functionals, and in particular the vortices arising in their
minimizers or critical points, have been studied intensively in
the mathematics literature. We refer in particular to the books
\cite{bbh} for $E_\ep$ and \cite{ss_book} for the functional with
magnetic field. The interested reader can find there more
information on the  physical and mathematical  background.\medskip

 We are interested in proving lower bounds on $F_\ep$, and in
particular estimates which relate $F_\ep(u,A)$ and $\pqnorm{\nabla_A u}{2}{\infty}$, the norm of $\nab_A u$ in the
Lorentz space $\lti$. Noticeably, Lorentz spaces were already used
in the context of the  Ginzburg-Landau energy by Lin and
Rivi\`ere in \cite{lr}. Their goal there was to  study energy
critical points in 3 dimensions, but what they used was interpolation ideas
 and the duality between Lorentz spaces $L^{2,1}$ and $\lti $.

The Ginzburg-Landau energy is generally unbounded as $\ep \to 0$;
 it blows up roughly like $\pi n \lep$, where $n$ is the
number (or total degree) of vortices. Our investigation of
estimates for $\pqnorm{\nabla_A u}{2}{\infty}$ is thus part of a quest for
intrinsic quantities in $\nab_A u$ which do not blow up as $\ep \to
0$, but rather remain of the order of $n$.

\subsection{Heuristics for idealized vortices}\label{idealv}

Let us now try to explain the interest and relevance of the
Lorentz space $\lti$ for this problem. The  space $\lti$, also
known as  ``weak-$L^2$", is a functional space
which is just ``slightly larger" than the Lebesgue space $L^2$.
One simple way of defining the $\lti$ norm is by
\begin{equation}
\label{ltinorm1} \pqnorm{f}{2}{\infty} = \sup_{\abs{E}<\infty}
\abs{E}^{-\hal} \int_E \abs{f(x)}\, dx,
\end{equation}
where $\abs{E}$ denotes the Lebesgue measure of $E$. An
equivalent way is through the super-level sets of $f$:
\begin{equation} \label{ltinorm2} \pqnorm{f}{2}{\infty}=\sup_{t>0} t \lambda_f(t)^\hal,\end{equation}
 where
$\lambda_f(t) = \abs{\{ x\in \om \;\vert\; |f(x)|>t\}}.$ For more
information on Lorentz spaces we refer for example to
\cite{grafakos,stein_intro}. A simple application of the
Cauchy-Schwarz inequality in (\ref{ltinorm1}) allows to check that
if $f$ is in $L^2$ then it is in $\lti$ with
$\pqnorm{f}{2}{\infty} \le \pnorm{f}{2}$.\medskip

Let us now consider vortices of  a complex-valued function $u$ in
the context of Ginzburg-Landau.  In the regime of small $\ep$, $u$
can have zeroes, but because  of the strong penalization of the term $\io
 (1-\abs{u}^2)^2$, $\abs{u}$ can be small only in (small) regions of
characteristic size $\ep$.

 Then around a zero at a point $x_0$,
  $u$ has a degree defined as the
 topological degree of $u/\abs{u}$ as a map from a circle to $\si$, or in other words
\begin{equation}
\label{degredefintro} d= \frac{1}{2\pi} \int_{\p B(x_0,r)}
\frac{\p }{\p \tau}\( \frac{u}{\abs{u}}\)\in \mz,\end{equation} where $r$ is
sufficiently small. One can describe the situation very roughly as
follows: $\abs{u}$ is small in a ball of radius $C\ep$, and $\abs{u}\approx
1$ outside of this ball, say in an annulus $B(x_0,R) \backslash
B(x_0, C\ep)$. The size of $R$ is meant to account for possible
neighboring zeroes. In this annulus, the model case is that of a
radial vortex of degree $d$, i.e \begin{equation}
\label{vortexradial} u(r, \theta) = f(r)e^{i d
\theta},
\end{equation}
where $(r,\theta)$ are the polar coordinates
centered at $x_0$, and $f$ is a real-valued function, close to $1$
in $B(x_0,R) \backslash B(x_0, C \ep)$. When computing the $L^2 $
norm of $\nab u$, we find that $\abs{\nab u}\approx \frac{\abs{d}}{r} $ in
the annulus and thus, using polar coordinates, \begin{eqnarray}
\nonumber
 \pnormspace{\nabla u}{2}{B(x_0,R)}^2 &\ge &  \int_{B(0,R) \backslash
B(0,C\ep)} \abs{\frac{d}{r}}^2  = \int_{C\ep}^R  \frac{2\pi  d^2}{r} \, dr\\
\label{l2norm} &\ge & 2\pi d^2 \log \frac{R}{C\ep}.\end{eqnarray}
This tells us that the (square of the) $L^2 $ norm of $\nab u$
blows up like $2\pi d^2 \lep$ as $\ep \to 0$. This is a crucial
fact in the analysis of Ginzburg-Landau, much used since
\cite{bbh}. Jerrard \cite{j} and Sandier \cite{sa} showed that
this picture is actually accurate even for arbitrary
configurations: without assuming that the vortex profile is
radial, the inequality (\ref{l2norm}) still holds (the radial
profile is actually the one that is minimal for the $L^2$ norm).
Moreover, any configuration with an  arbitrary number of vortices
can be understood as many such annuli, possibly at very close
distance to each other, glued together.  Good lower bounds like
(\ref{l2norm}) can be added up together by keeping annuli with the
same conformal type. This was the basis of the ``vortex-balls
construction" that they formulated and  which was used extensively
to understand Ginzburg-Landau minimizers, in particular in
\cite{ss_book}.\medskip

On the other hand, let us calculate (roughly) the $\lti$ norm of
$\nab u$ for the above vortex. We recall that $\abs{\nab u}\approx
\frac{\abs{d}}{r}$ in the annulus $B(x_0,R) \backslash B(x_0, C\ep)$.
Using the definition (\ref{ltinorm2}), we have $ \abs{\nab u}>t$ if
and only if $ r<\abs{d}/t $. Thus
$$\lambda_{\abs{\nab u} }(t)\approx \pi d^2/t^2, $$ and we find
\begin{equation}
\label{ltinormvid} \|\nab u\|_{\lti (B(x_0,R) \backslash
B(x_0,C\ep) ) } \approx \sqrt{\pi} |d|.\end{equation} So in contrast,
the $\lti$ norm  of $\nab u$ {\it does not blow up} as $\ep \to
0$. One can see that this space is critical in the sense that
$1/\abs{x}$ (barely) fails to be in  $L^2 $  or in $L^{2,q}$ for any $q<\infty$ (its norm blows up logarithmically in all cases)
 but  is in $\lti$ and in all $L^p$ for $p<2$.

Moreover, from this formula (\ref{ltinormvid}), it is expected that the $L^{2,\infty}$ norm
can serve to estimate the total degree $\sum \abs{d_i}$ of all the
vortices of a configuration. This is convenient since the total
degree $\sum \abs{d_i}$ is generally obtained via a ``ball
construction" that is nonunique. On the other hand $\pqnorm{\nabla u}{2}{\infty}$ provides a unique and intrinsic quantity useful to
evaluate the number  of vortices.\medskip

Because of these remarks and because of the paper \cite{lr},
it could be expected that  Lorentz spaces are a suitable functional setting in which to study Ginzburg-Landau vortices.
One may point out that there are other spaces that
would be critical for  the profile $1/\abs{x}$, such as Besov spaces;
however, it seems difficult to find an effective way of using them in
connection with the Ginzburg-Landau  energy.

 The main goal of our results is to  give a rigorous
basis to the above observations. The connection with the  Lorentz
norm of $\nab u$ is made through the ``vortex-balls
construction" of Jerrard and Sandier, as formulated in
\cite{ss_book}. Our estimates will in fact provide an improvement
of these lower bounds by adding an extra positive term in the
lower bounds, which is then related to the Lorentz norm. Just as
in  the ball construction method, one of the interests of the
result is that it is valid under very few assumptions: only a very
weak upper bound on the energy, even when $u$ has a large number
of vortices, unbounded as $\ep \to 0$. This creates serious
technical difficulties but is important since such situations
occur for energy minimizers when there is a large applied magnetic
field, as proved in \cite{ss_book}.

\subsection{Main results}
Let us point out that the estimates we prove are   not on the
Lorentz norm of $\nab u$ but rather on that of $\nab_A u$. The
reason is that the energy $F_\ep$ is {\it gauge-invariant} : it
satisfies  $F_\ep(u,A) = F_\ep(u e^{i\Phi}, A + \nab \Phi)$ for
any smooth function $\Phi$. Thus the quantity $\abs{\nab u}$ is not a
gauge-invariant quantity, hence not an intrinsic physical
quantity. This is why it is  replaced by the gauge-invariant
``covariant derivative" $\abs{\nab_A u}$.

Our method consists in proving the following improvement of the
``ball construction" lower bounds (see \cite{ss_book}, Chapter 4):
\begin{thm}[Improved lower bounds]
\label{energy_bound} Let $\alpha \in (0,1)$.  There exists
$\varepsilon_0 >0$ (depending on $\alpha$) such that for
$\varepsilon \le \varepsilon_0$ and $u, A$ both $C^1$ such that
$F_{\varepsilon}(\abs{u},\Omega) \le \varepsilon^{\alpha-1}$,  the
following hold.

For any $1 > r > C \varepsilon^{\alpha/2}$, where $C$ is a
universal constant, there exists a finite, disjoint collection of
closed balls, denoted by $\mathcal{B}$, with the following
properties.

1. The sum of the radii of the balls in the collection is $r$.

2.  Defining $\Omega_\ep= \{ x \in \Omega \, | \,\dist(x, \bo) >
\ep\}$, we have

$\{x\in \Omega_{\varepsilon} \; | \; \abs{u(x)-1} \ge \delta \}
\subset V := \Omega_{\varepsilon} \cap \left(\cup_{B\in
\mathcal{B}} B\right)$, where $\delta = \varepsilon^{\alpha/4}$.

3. We have
\begin{equation}\label{e_b_15}
\begin{split}
&\frac{1}{2} \int_V \abs{\nabla_A u}^2 + \frac{1}{2\varepsilon^2}(1-\abs{u}^2)^2 + r^2(\curl{A})^2 \\
& \ge \pi D \left( \log{\frac{r}{\varepsilon D}} -C  \right) +
\frac{1}{18} \int_{V} \abs{\nabla_{A+G} u}^2 +
\frac{1}{2\varepsilon^2}(1-\abs{u}^2)^2,
\end{split}
\end{equation}
where $G$ is some explicitly constructed vector field, $d_B$
denotes $\deg(u, \p B)$ if $B \subset \Omega_\varepsilon$ and $0$
otherwise,
\begin{equation*}
D = \sum_{\substack{B\in \mathcal{B} \\ B \subset
\Omega_\varepsilon}} \abs{d_B}
\end{equation*}
is assumed to be nonzero, and $C$ is universal.
\end{thm}
The improvement with respect to Theorem 4.1 in \cite{ss_book} is
the addition of the extra term $\frac{1}{18}\int\abs{\nab_{A+G}
u}^2$. The term $G$ is a  vector-field constructed in the course
of the ball construction, which essentially compensates for the
expected behavior of  $\nab_A u$ in the vortices.  One can take it
to be $\tau d/r$ in every annulus of the ball construction  where
$u$ has a constant  degree $d$,  $\tau $ denotes the unit tangent
vector to each circle centered at $x_0$, the center of the
annulus, and $r= \abs{x-x_0}$. By extending $G$ to be zero outside
of the union of balls $V$, we easily deduce:
\begin{cor} Let $(u,A)$ be as above, then
\begin{equation}
\label{vjs} \io |\nab_A u- iG u|^2 \le C \( F_\ep(u,A) - \pi D
\log \( \frac{r}{\ep D}  - C\) \)
\end{equation}
 where $G$ is the explicitly constructed vector field of Theorem
 \ref{energy_bound},  and  $C$ a universal constant.\end{cor}
 The right-hand side of
this inequality can be considered as the ``energy-excess",
difference between the total energy and the expected vortex energy
provided by the ball construction lower bounds.  Thus we control
$\io |\nab_A u -iG u|^2$ by the energy-excess. This fact is used 
repeatedly in the sequel paper \cite{p2} to better understand the
behavior of $\nab_A u$ for minimizers and almost minimizers of the
Ginzburg-Landau  energy with applied magnetic field.

One can also note that such a control (\ref{vjs}) has a similar
flavor to a result of Jerrard-Spirn \cite{jspirn} where they
control the difference (in a weaker norm but with better control) of the Jacobian
of $u$ to a measure  of the form $\sum d_i \delta_{a_i}$ by the
energy-excess. \medskip

Once Theorem \ref{energy_bound} is proved, we turn to obtaining an
$\lti $ estimate  from which $G$ has disappeared. In order to do
so,  we can bound below $\pnorm{\nab_{A+G} u}{2}$ by
$\pqnorm{\nabla_{A+G} u}{2}{\infty}$; the more delicate task is
then to control $\pqnorm{G}{2}{\infty}$ in a way that only depends
on the final data of the theorem, that is on the degrees of the
final balls constructed above and on the energy. This task is
complicated by the possible presence of large numbers of vortices
very close to each other, and compensations of vortices of large
positive degrees with vortices of large negative degrees. To
overcome this, $G$ is not defined exactly as previously said, but
in a modified way, and $\pqnorm{G}{2}{\infty}$ is controlled not
only through the degrees but also through  the total energy.

We then arrive at the following main result :
\begin{thm}[Lorentz norm bound]
\label{norm_switch} Assume the hypotheses and results of Theorem
\ref{energy_bound}.  Then there exists a universal constant $C$
such that
\begin{multline}\label{n_s_1}
\frac{1}{2} \int_V \abs{\nabla_A u}^2 +
\frac{(1-\abs{u}^2)^2}{2\varepsilon^2}
 + r^2(\curl{A})^2 + \pi \sum  \abs{d_{B}}^2\\
\ge C \pqnormspace{\nabla_A u}{2}{\infty}{V}^2 + \pi \sum \abs{d_B}
\left( \log{\frac{r}{\varepsilon \sum \abs{d_B}}} -C  \right),
\end{multline}
where the sums are taken over all the balls $B$ in the final
collection $\mathcal{B}$ that are included in $\Omega_\ep$.
\end{thm}

This theorem bounds below the energy contained in the union of
balls $V$
 in terms of the $\lti$ norm on $V$.  It is a simple matter to extend
  these estimates to all of $\Omega$, and deduce a control of the $\lti$ norm of $\nab_A u$
  by the energy-excess, plus the term $\sum \abs{d_B}^2$.
   This is the content of the following corollary.

\begin{cor}\label{coro1} Assuming the hypotheses and results of
Theorem \ref{energy_bound}, there exists a universal constant $C$
such that \begin{equation} \pqnormspace{\nabla_A u}{2}{\infty}{\Omega}^2 \le C \(
F_\ep(u,A) - \pi \sum \abs{d_B} \log \frac{r}{\ep \sum \abs{d_B}} + \sum
\abs{d_B}^2\),\end{equation}
where the sums are taken over all the balls $B$ in the final
collection $\mathcal{B}$ that are included in $\Omega_\ep$.
\end{cor}

These estimates can indeed help to bound from above
$\pqnormspace{\nabla_A u}{2}{\infty}{\Omega}^2$ by the total
number of vortices, provided we can control the energy-excess by
that number of vortices. This can in turn serve to obtain stronger
convergence results when a weak limit of $\nab_A u$ is known. For
example, if one considers  the energy $E_\ep$ (which we recall
amounts to  setting $A\equiv 0$), it is known from  Bethuel-Brezis-H\'elein \cite{bbh} that
$\pi \sum \abs{d_B} \lep = \pi n \lep$ is the leading order of the
energy (at least for minimizers) and that the next order term is a
term of order 1, called the ``renormalized energy" $W$, that
accounts for the interaction between the vortices. The upper bound
of Corollary \ref{coro1} roughly tells us that
$$\pqnormspace{\nabla u}{2}{\infty}{\Omega}^2 \le C (W + \sum \abs{d_B}^2 + \sum \abs{d_B} \log \sum \abs{d_B}). $$
It is expected that the total cost of interaction  of the vortices
in $W$ is of order of  $n^2$, where $n= \sum \abs{d_B}$ is the total vorticity mass  (here $n$ can blow up as $\ep \to 0$). Thus, we
obtain a bound of the form
$$\pqnormspace{\nabla u}{2}{\infty}{\Omega}^2 \le C n^2, $$ which indeed bounds the $\lti$
norm of $\nab u$ by an order of $n $, the total vorticity mass,
as expected in the heuristic calculations of Section \ref{idealv}.

In the simplest case where  we know that $E_\ep(u_\ep) \le \pi n
\lep + C$, which happens for energy minimizers when $n$ is
bounded, as proved in \cite{bbh},  we then  deduce that $\|\nab u
\|_{\lti} \le C$.  To be more precise, for the minimizers of
$E_\ep$ found in \cite{bbh}, we have
\begin{prop}[Application to minimizers of $E_\ep$ with Dirichlet
boundary conditions]\label{corobbh}\indent\par\noindent
 Let $\Omega$ be starshaped and $u_\ep$ minimize $E_\ep$ under the constraint $u_\ep =g $ on
$\bo$, where $g$ is a fixed $\si$-valued map of degree $d>0$ on the
boundary of $\Omega$, as studied in \cite{bbh}. Then there exists
a  universal constant $C$ such that
$$\pqnormspace{\nabla u_\ep}{2}{\infty}{\Omega}^2 \le C(\min_{\om^d} W + d (\log{d}+ 1)  )+o_\ep(1).$$
Moreover, as $\ep \to 0$, $$\nab u_\ep \rightharpoonup \nab u_\star
\quad \text{weakly-$*$ in }\ \lti(\om),$$ where $u_\star $ is the
$S^1$-valued ``canonical harmonic map" of \cite{bbh} to which
converges $u$ in $C^k_{loc}$ outside of a set of $d$ vortex
points.
\end{prop}
Note that the renormalized energy $W$ depends on $g$ (hence on
$d$), and the $d \log d$ is not optimal here; rather, it should be
$d$. It is more delicate to obtain this kind of improvement to the
estimate; this is one of the things done in \cite{p2} in the
context of the energy with applied magnetic field. Also the
convergence of $\nab u_\ep$ {\it cannot be strengthened},
convergence in $\lti$ strong does not hold, as illustrated by the
following model case: let $V_\ep $ be the vector field
$\frac{(x-p_\ep)^\perp}{|x-p_\ep|^2}$ and $V=
\frac{(x-p)^\perp}{|x-p|^2} $ with $p_\ep \to p$ as $\ep \to 0$.
Then $2\sqrt{\pi} \le \|V_\ep - V\|_{\lti} \le 4\sqrt{\pi}$, while clearly $V_\ep
\rightharpoonup V$ weakly-$*$ in $\lti$.
\medskip

We have focused on proving upper bounds on $\pqnorm{\nabla_A
u}{2}{\infty}$ in terms of its $L^2 $ norm and Ginzburg-Landau
energy. It is not difficult to obtain  some adapted, though not
optimal, lower bounds. For example, we can prove the following:
\begin{prop}\label{proplb}
 Let $f\in L^{\infty}(\Omega)$ be such that $\pnormspace{f}{\infty}{\Omega}
\le \frac{C}{\ep}$ for some $\ep <1$. Then
\begin{equation}
\label{borninf} \pqnormspace{f}{2}{\infty}{\Omega}^2 \ge \frac{1}{2\lep} \io
\abs{f}^2 - \frac{ C^2 \abs{\Omega}}{2\lep} .\end{equation}
\end{prop}
This proposition is a direct consequence of the definition of the
$\lti$ norm. Its short proof is presented in Section
\ref{l_2_inf_discussion}.

For critical points of the Ginzburg-Landau energy, it is known
that the gradient bound $\pnormspace{\nabla_A u}{\infty}{\Omega}\le
\frac{C}{\ep}$ holds. Thus applying Proposition \ref{proplb} to
$f= \nab_A u$, we find
$$\pqnormspace{\nabla_A u}{2}{\infty}{\Omega}^2  \ge \frac{1}{2\lep} \io \abs{\nab_A u}^2
-o(1).$$ Knowing some  lower bounds (provided by the ball
construction)  of the type $\io\abs{\nab_A u}^2 \ge 2\pi n \lep$, where
$n$ is the total degree of the vortices, we find lower bounds of
the type  $\pqnormspace{\nabla_A u}{2}{\infty}{\Omega}^2 \ge \pi n $, also
relating the $\lti$ norm of $\nab_A u$ to the total number
 of vortices.\medskip

In \cite{p2}, which is the sequel of this paper, the ideas and main 
results of this paper are extended to the case of the full
Ginzburg-Landau energy with an applied magnetic field, getting
better estimates on $\pqnormspace{\nabla_A u}{2}{\infty}{\Omega}$
in terms of the number of vortices.  These results lead to a somewhat stronger (than previously known results) convergence of  $\nab_A u$ and of the Jacobian determinants of $u$  when certain energy conditions are fulfilled.

\subsection{Plan}
The paper is organized as follows: in Section 2, for the
convenience of the reader, we  give a review (with slight
modifications) of the crucial definitions and ingredients for the
vortex-balls construction following Chapter 4 of \cite{ss_book}.

In Section 3 we present the main argument, with the introduction
of the function $G$ and the ``trick'' that allows us to gain an extra term in the lower
bounds for the energy on annuli.

In Section 4  we show how this extra term incorporates into the
estimates through the growing and merging of balls, and hence through the
whole ball construction.

In Section 5  we deduce the proof of the main results.

In Section 6 we estimate the $\lti$ norm of $G$ in order to pass
from Theorem \ref{energy_bound} to Theorem \ref{norm_switch}. This
is the only section in which $\lti$ comes into play.

In Section 7 we show how the methods of this paper can be adapted
to work with the version of the ball construction formulated by
Jerrard in \cite{j}, at the expense of less control of
$\|G\|_{\lti}$.

\section{Reminders for the vortex balls construction}
\subsection{The ball growth method}

In finding lower bounds for the Ginzburg-Landau energy of
a configuration $(u,A)$ it is most convenient to work on annuli,
the deleted interior discs of which contain the set where $u$ is near $0$,
 and in particular the vortices.  On each annulus, a lower bound is found
 in terms of a topological term (the degree of the vortex) and a conformal
 factor, which we define to be the logarithm of the ratio of the outer and
 inner radii of the annulus.  Therefore, to create useful lower bounds we must
  be able to identify the set where $u$ is near $0$ and then create a family
  of annuli with large conformal type outside this set.  The first component
  of the process uses energy methods to find a covering of the set by small,
  disjoint balls, and is addressed later.  The second component is known as
  the general ball growth method and is presented in this section.  Here we
  follow the construction of Chapter 4 from \cite{ss_book}.

As a technical tool we will need the ability to merge two tangent or overlapping
balls into a single ball that contains the original balls, and with the property
 that its radius is equal to the sum of the radii of the original balls.
  Our first lemma recalls how to do such a merging.  We write $r(B)$ for
  the radius of a ball $B$.

\begin{lem}\label{ball_merging}
Let $B_1$ and $B_2$ be closed balls in $\Rn{n}$ such that $B_1 \cap B_2 \ne \varnothing$.  Then there is a closed ball $B$ such that $r(B) = r(B_1) + r(B_2)$ and $B_1 \cup B_2 \subset B$.
\end{lem}
\begin{proof}
If $B_1 = B(a_1,r_1)$ and $B_2 = B(a_2,r_2)$, then $ B =
B\left(\frac{r_1 a_1 + r_2 a_2}{r_1+r_2},r_1+r_2\right) $ has the
desired properties.
\end{proof}

The ball growth lemma now provides an algorithm for growing an
initial collection of small balls into a final collection of large
balls.  Essentially, the balls in a collection are grown
concentrically by increasing their radii by the same conformal
factor.  This is continued until a tangency occurs, at which point
the previous lemma is used to merge the tangent balls.  The
process is then repeated in stages until the collection is of the
desired size.  The annuli of interest at each stage are formed by
deleting the initial collection of balls from the final
collection; the construction guarantees that all of the annuli in
a stage have the same conformal type.

Given a finite collection of disjoint balls, $\mathcal{B}$, we define the radius of the collection, $r(\mathcal{B})$, to be the sum of the radii of the balls in the collection, i.e.
\begin{equation*}
 r(\mathcal{B}) = \sum_{B \in \mathcal{B}} r(B).
\end{equation*}
For any $\lambda >0$ and any ball $B = B(a,r),$ we define
$\lambda B = B(a,\lambda r)$.  Extending this notation to
 collections of balls, we write $\lambda \mathcal{B} =
  \{ \lambda B \;\vert\; B \in \mathcal{B} \}$.  For an annulus
   $A =B(a,r_1) \backslash B(a,r_0)$, we define the conformal
   factor by $\tau = \log(r_1/r_0)$.  We can now state the ball growth lemma, the proof of which can be found in Theorem 4.2 of \cite{ss_book}.

\begin{lem}[Ball growth lemma]\label{ball_growth_lemma}
Let $\mathcal{B}_0$ be a finite collection of disjoint, closed balls.  There exists a family $\{ \mathcal{B}(t) \} _{t\in \Rn{}_{+}}$ of collections of disjoint, closed balls such that the following hold.

1. $\mathcal{B}_0 = \mathcal{B}(0)$.

2. For $s\ge t \ge 0$,
\begin{equation*}
\bigcup_{B \in \mathcal{B}(t)}B \subseteq \bigcup_{B \in \mathcal{B}(s)}B.
\end{equation*}

3. There exists a finite set $T \subset \Rn{+}$ such that if $[t,s]\subset \Rn{+} \backslash T$, then $\mathcal{B}(s) = e^{s-t}\mathcal{B}(t)$.  In particular, if $B(s)\in \mathcal{B}(s)$ and $B(t) \in \mathcal{B}(t)$ are such that $B(t) \subset B(s)$, then $B(s) = e^{s-t} B(t)$ and the conformal factor of the annulus $B(s) \backslash B(t)$ is $\tau = s-t$.

4. For every $t\in \Rn{+}$, $r(\mathcal{B}(t))=e^tr(\mathcal{B}_0)$.
\end{lem}

We now show how to couple lower bounds to the geometric construction.  We may think of a function $\mathcal{F}:\Rn{2}\times \Rn{+}\rightarrow \Rn{+}$ as being defined also for collections of balls, $\mathcal{B}$, via the identifications
\begin{equation*}
 \mathcal{F}(B(x,r)) = \mathcal{F}(x,r)
\end{equation*}
and
\begin{equation*}
 \mathcal{F}(\mathcal{B}) = \sum_{B \in \mathcal{B}} \mathcal{F}(B).
\end{equation*}
Here and for the rest of the paper we employ the notation $\bar{B}$ to refer to a specific ball $\bar{B}$ in some collection, and not to refer to the closure of $B$.  We will also abuse notation by writing $\bar{B} \cap \mathcal{B}(t)$ for the collection $\{\bar{B} \cap B \;\vert\; B \in \mathcal{B}(t)\}$.

\begin{lem}\label{function_balls}
Let $\mathcal{B}_0$ be a finite collection of disjoint, closed balls, and suppose that $\mathcal{B}(t)$ is the collection of balls obtained from $\mathcal{B}_0$ by growing them according to the ball growth lemma.  Fix a time $s>0$ and suppose that $0 < s_1 < \dotsb < s_K \le s$ denote the times at which mergings occur in the the ball growth lemma, i.e. let the $s_i$ be an increasing enumeration of the set $T$ defined there.  Then
\begin{equation}\label{fn_b_1}
\mathcal{F}(\mathcal{B}(s)) - \mathcal{F}(\mathcal{B}_0)  =
\int_0^s \sum_{B(x,r)\in\mathcal{B}(t)} r \frac{\partial
\mathcal{F}}{\partial r}(x,r)\, dt + \sum_{k=1}^{K}
\mathcal{F}(\mathcal{B}(s_k))-\mathcal{F}(\mathcal{B}(s_k))^{-},
\end{equation}
where $\mathcal{F}(\mathcal{B}(s_k))^{-}=\lim\limits_{t\rightarrow s_k^-} \mathcal{F}(\mathcal{B}(t))$.  Moreover, for any $\bar{B} \in \mathcal{B}(s)$, the following localized version of \eqref{fn_b_1} holds:
\begin{equation}\label{fn_b_2}
\mathcal{F}(\bar{B}) - \mathcal{F}(\bar{B} \cap \mathcal{B}_0 ) =
\int_0^s \sum_{B(x,r)\in \bar{B} \cap \mathcal{B}(t)} r
\frac{\partial \mathcal{F}}{\partial r}dt + \sum_{k=1}^{K}
\mathcal{F}(\bar{B} \cap \mathcal{B}(s_k) )-\mathcal{F}(\bar{B} \cap
\mathcal{B}(s_k))^{-}.
\end{equation}
\end{lem}
\begin{proof}
The proof is the same as in Proposition 4.1 of \cite{ss_book}, but here we keep the second sum in \eqref{fn_b_1} rather than bounding it.
\end{proof}

Note that in the case that
\begin{equation*}
\mathcal{F}(x,r) = \int_{B(x,r)} e(u)
\end{equation*}
for some $u$-dependent energy density $e(u)$, the first term on the right of \eqref{fn_b_1} corresponds to integration in polar coordinates on each annulus, and the second corresponds to the energy contained in the non-annular parts of $\mathcal{B}(s)$.

\subsection{The radius of a set}

In order to effectively use the ball growth lemma to generate
lower bounds, it is necessary to first produce a collection of
disjoint balls covering the set where $u$ is near $0$.  We do this
by using the concept of the radius of a set, which is useful in
two ways.  First, it is defined as an infimum over all coverings
of the set by collections of balls, so that by exceeding the
infimum we may find a covering of the set by balls.  Second, it is
comparable to the $\mathcal{H}^1$ Hausdorff measure of the boundary, and so it can
be used with the co-area formula to produce coverings by balls of
the set where $\abs{u}$ is far from unity.

We define the radius of a compact set $\omega \subset \Rn{2}$, written $r(\omega)$, by
\begin{equation*}
r(\omega) = \inf \{r(B_1)+\dotsb+r(B_k) \;\vert\; \omega \subset \cup_{i=1}^k B_i \text{ and }k< \infty    \}.
\end{equation*}
We make the following remarks.\\
1) In the definition we may assume that the balls are disjoint.  If they are not, then we merge balls that meet into a single ball with radius equal to the sum of the radii of the merged balls according to Lemma \ref{ball_merging}.  \\
2) If $A\subseteq B$ then $r(A) \le r(B)$.\\
3) The infimum is not necessarily achieved.

It is necessary to also introduce a modification of the radius that measures the radius of the connected components of a compact set $\omega$ that lie inside an open set $\Omega$.  Indeed, we define
\begin{equation*}
    r_{\Omega}(\omega) = \sup \{ r(K\cap \omega) \;|\; K\subset
\Omega \text{ s.t. } K \text{ is compact and }\partial K \cap \omega = \varnothing  \}.
\end{equation*}
The following lemmas record the crucial properties of these quantities.  The omitted proofs may be found in Section 4.4 of \cite{ss_book}.

\begin{lem}
    Let $\omega$ be a compact subset of $\Rn{2}$.  Then
\begin{equation}
2r(\omega) \le \mathcal{H}^1(\partial\omega).
\end{equation}
\end{lem}

\begin{lem}\label{haus_bound}
    Let $\Omega$ be open and $\omega \subset \Omega$ be a compact set.  Then
\begin{equation}
2 r_{\Omega}(\omega) \le \mathcal{H}^1(\partial\omega\cap\Omega).
\end{equation}
\end{lem}

\begin{lem}
    Let $\omega_1, \omega_2$ be compact subsets of $\Rn{2}$.  Then
\begin{equation}
    r(\omega_1 \cup \omega_2) \le r(\omega_1) + r(\omega_2).
\end{equation}
\end{lem}

\begin{lem}\label{rad_subadd}
    Let $\omega_1, \omega_2$ be compact sets, and let $\Omega \subset \Rn{2} $ be an open set.  Then
\begin{equation}
    r_{\Omega}(\omega_1 \cup \omega_2) \le r_{\Omega}(\omega_1) + r_{\Omega}(\omega_2).
\end{equation}
\end{lem}
\begin{proof}
If $\Omega \subset \omega_1 \cup \omega_2$, then the result is trivial.  Suppose otherwise.  Let $K \subset \Omega$ be such that $K$ is compact and $\partial K \cap (\omega_1 \cup \omega_2) = \varnothing$.  Then $(\partial K \cap \omega_1) \cup (\partial K \cap \omega_2) = \varnothing$, which implies that $\partial K \cap \omega_1 =\varnothing$ and $\partial K \cap \omega_2 =\varnothing$.  Hence,
\begin{equation}
\begin{split}
r(K \cap (\omega_1 \cup \omega_2)) &= r((K\cap \omega_1)\cup(K\cap \omega_2)) \\
&\le r(K\cap \omega_1) + r(K\cap \omega_2) \\
&\le r_{\Omega}(\omega_1) + r_{\Omega}(\omega_2).
\end{split}
\end{equation}
Taking the supremum over all such $K$, we get $r_{\Omega}(\omega_1 \cup \omega_2) \le r_{\Omega}(\omega_1) + r_{\Omega}(\omega_2).$
\end{proof}

We will now use these concepts to compare the energy of a real-valued function $\rho$, defined on an open set $\Omega$, to the radius of the set where $\rho$ is far from unity.

\begin{lem}\label{e_rad_bound}
Let $\rho \in C^1(\Omega,\mathbb{R})$ with $\Omega \subset \Rn{2}$
open and bounded.  Let
\begin{equation}
F_{\varepsilon}(\rho,\Omega) = \frac{1}{2} \int_{\Omega} \abs{\nabla\rho}^2 + \frac{1}{2\varepsilon^2} (1-\rho^2)^2.
\end{equation}
Then there is a universal constant $C$ such that
\begin{equation}
    r_{\Omega}\left(\left\{\rho \le 1/2 \right\} \cup \left\{ \rho \ge 3/2 \right\}\right) \le \varepsilon C F_{\varepsilon}(\rho,\Omega).
\end{equation}
\end{lem}
\begin{proof}
By the Cauchy-Schwarz inequality and the co-area formula we have that
\begin{equation}\label{erb_1}
\begin{split}
F_{\varepsilon}(\rho,\Omega) & = \frac{1}{2} \int_{\Omega} \abs{\nabla\rho}^2 + \frac{1}{2\varepsilon^2} (1-\rho^2)^2 \\
& \ge \frac{1}{\varepsilon\sqrt{2}} \int_{\Omega} \abs{\nabla \rho}\abs{1-\rho^2} \\
&= \frac{1}{\varepsilon\sqrt{2}} \int_0^{\infty} \int_{\{\rho=t\}\cap \Omega} \abs{1-\rho^2} d\mathcal{H}^1 dt\\
&= \frac{1}{\varepsilon\sqrt{2}} \int_0^{\infty} \abs{1-t^2} \mathcal{H}^1(\{\rho=t\}\cap \Omega)dt.
\end{split}
\end{equation}
We break the last integral into two parts and bound
\begin{equation}\label{erb_2}
\begin{split}
&\frac{1}{\varepsilon\sqrt{2}} \int_0^{\infty} \abs{1-t^2} \mathcal{H}^1(\{\rho=t\}\cap \Omega)dt\\
&\ge \frac{1}{\varepsilon\sqrt{2}} \int_{\frac{1}{2}}^{\frac{3}{4}} (1-t^2) \mathcal{H}^1(\{\rho=t\}\cap \Omega)dt +
\frac{1}{\varepsilon\sqrt{2}} \int_{\frac{5}{4}}^{\frac{3}{2}} (t^2-1) \mathcal{H}^1(\{\rho=t\}\cap \Omega)dt \\
&= \frac{1}{\varepsilon 4 \sqrt{2}} (1-t_0^2)\mathcal{H}^1(\{\rho=t_0\}\cap \Omega) +
\frac{1}{\varepsilon 4\sqrt{2}} (t_1^2-1)\mathcal{H}^1(\{\rho=t_1\}\cap \Omega),
\end{split}
\end{equation}
where the last equality follows from the mean value theorem, and $t_0 \in (\frac{1}{2},\frac{3}{4})$ and $t_1 \in (\frac{5}{4},\frac{3}{2})$.  The bounds on $t_0$ and $t_1$ imply that
\begin{equation}\label{erb_3}
\begin{split}
& (1-t_0^2) \ge  1 -\frac{9}{16} = \frac{7}{16} \text{, and}\\
& (t_1^2-1) \ge \frac{25}{16} -1 = \frac{9}{16}.
\end{split}
\end{equation}
Combining \eqref{erb_1}, \eqref{erb_2}, and \eqref{erb_3}, we get
\begin{equation}\label{erb_4}
\begin{split}
F_{\varepsilon}(\rho,\Omega) &\ge \frac{7}{\varepsilon 64 \sqrt{2}} \mathcal{H}^1(\{\rho=t_0\}\cap \Omega) + \frac{9}{\varepsilon 64 \sqrt{2}} \mathcal{H}^1(\{\rho=t_1\}\cap \Omega)\\
& \ge \frac{7}{\varepsilon 64 \sqrt{2}} \left(\mathcal{H}^1(\{\rho=t_0\}\cap \Omega) +\mathcal{H}^1(\{\rho=t_1\}\cap \Omega)\right).
\end{split}
\end{equation}
Write $S_{t_0}$ and $S^{t_1}$ for the $\Rn{2}$-closures of the sets $\{ x\in \Omega \;\vert\; \rho(x) \le t_0\}$ and $\{ x\in \Omega \;\vert\; \rho(x) \ge t_1\}$ respectively.  The bounds  $t_0 \ge \frac{1}{2}$, $t_1 \le \frac{3}{2}$ imply the inclusions $\{ \rho \le 1/2\} \subset S_{t_0}$ and $\{\rho \ge 3/2\} \subset S^{t_1}$.  We may then apply lemmas \ref{haus_bound} and \ref{rad_subadd}  to find the bounds
\begin{equation}\label{erb_5}
\begin{split}
\mathcal{H}^1(\{\rho=t_0\}\cap \Omega)
+\mathcal{H}^1(\{\rho=t_1\}\cap \Omega)
&=\mathcal{H}^1(\partial S_{t_0}  \cap \Omega) +\mathcal{H}^1(\partial S^{t_1} \cap \Omega) \\
&\ge 2 r_{\Omega}(S_{t_0}) + 2 r_{\Omega}(S^{t_1}) \\
&\ge 2 r_{\Omega}\left(\left\{\rho \le 1/2 \right\}\right) + 2 r_{\Omega}\left(\left\{ \rho \ge 3/2 \right\}\right) \\
&\ge 2 r_{\Omega}\left(\left\{\rho \le 1/2 \right\} \cup \left\{ \rho \ge 3/2 \right\}\right).
\end{split}
\end{equation}
Putting \eqref{erb_5} into \eqref{erb_4} yields the desired estimate with $C = \frac{32\sqrt{2}}{7}$.
\end{proof}

\section{Improved lower bounds on annuli}

In this section we will show how to obtain lower bounds for the Ginzburg-Landau energy in terms of the degree.  We begin by constructing estimates on circles.  The primary difference between our estimates and those constructed previously is that we arrive at our lower bounds by introducing an auxiliary function $G$ and using a completion of the square trick.  This allows us to retain terms involving $G$ and thereby create an energy bound with a novel term.  Before properly defining $G$ let us prove the lower bounds on circles.

We first record a simple lemma (see for example Lemma 3.4 in
\cite{ss_book}).

\begin{lem}\label{polar}
Let $u\in H^1(\Omega,\mathbb{C})$ be written (at least locally) $u = \rho v$, where $\rho = \abs{u}$ and $v=e^{i\varphi}$.  Then $\abs{\nabla_A u }^2 = \abs{\nabla \rho}^2 + \rho^2 \abs{\nabla \varphi - A}^2 = \abs{\nabla \rho}^2 + \rho^2\abs{\nabla_A v}^2$.
\end{lem}

Now we prove the lower bounds on circles.

\begin{lem}\label{low_boundary}
Let $B:=B(a,r) \subset \Rn{2}$, and suppose that $v:\partial B \rightarrow \mathbb{S}^1$ and $A:B \rightarrow \Rn{2}$ are both $C^1$.  Let $G:\partial B \rightarrow \Rn{2}$ be given by $G=\frac{c \tau}{r}$, where $\tau$ is the oriented unit tangent vector field to $\partial B$ and $c$ is a constant.  Write $d_B :=\deg(v,\partial B)$. Then for any $\lambda >0$,
\begin{equation}\label{l_b_1}
\frac{1}{2}\int_{\partial B} \abs{\nabla_A v}^2 + \frac{\lambda}{2}\int_B (\curl{A})^2 \ge \frac{1}{2} \int_{\partial B} \abs{\nabla_{A+G} v}^2 + \frac{\pi}{r}(2cd_B-c^2) -\frac{\pi c^2}{2\lambda}.
\end{equation}
\end{lem}
\begin{proof}
Define the quantity
\begin{equation}\label{l_b_2}
    X := \int_B \curl{A} = \int_{\partial B} A\cdot \tau.
\end{equation}
We write $v = e^{i\varphi}$ and recall that $2 \pi d_B = \int_{\partial B} \nabla \varphi \cdot \tau$.
Using Lemma \ref{polar}, we see
\begin{equation}\label{l_b_3}
\begin{split}
\int_{\partial B} \abs{\nabla_{A+G} v}^2 &=  \int_{\partial B} \abs{\nabla \varphi -A -G}^2 \\
&=\int_{\partial B} \abs{G}^2 -2\int_{\partial B}G\cdot (\nabla \varphi -A) + \int_{\partial B}\abs{\nabla \varphi -A}^2\\
& = \frac{2\pi r c^2}{r^2} -\frac{2c}{r} \int_{\partial B} \nabla \varphi \cdot \tau
+ \frac{2c}{r}\int_{\partial B} A\cdot \tau + \int_{\partial B}\abs{\nabla_A v}^2 \\
& = \frac{2\pi c^2}{r} - \frac{2c}{r}2\pi d_B + \frac{2c}{r}X + \int_{\partial B}\abs{\nabla_A v}^2 \\
& = \frac{2\pi(c^2-2 c d_B)}{r} + \frac{2c}{r}X + \int_{\partial B}\abs{\nabla_A v}^2.
\end{split}
\end{equation}
An application of H\"{o}lder's inequality shows that
\begin{equation}\label{l_b_4}
\int_B (\curl{A})^2 \ge \frac{1}{\pi r^2} \left( \int_B \curl{A} \right)^2 = \frac{1}{\pi r^2} X^2.
\end{equation}
Combining \eqref{l_b_3} and \eqref{l_b_4} yields the inequality
\begin{multline}\label{l_b_5}
 \frac{1}{2}\int_{\partial B} \abs{\nabla_A v}^2 + \frac{\lambda}{2}\int_B (\curl{A})^2
 \ge \frac{1}{2} \int_{\partial B}\abs{\nabla_{A+G} v}^2 + \frac{\pi(2 c d_B-c^2)}{r} - \frac{c}{r}X + \frac{\lambda}{2\pi r^2}X^2.
\end{multline}
As $X$ varies, the minimum value of the right hand side occurs when $X=\frac{\pi c r}{\lambda}$.  Plugging this into \eqref{l_b_5} yields \eqref{l_b_1}.

\end{proof}

For this lemma to be useful we must construct a function $G:
\Omega \rightarrow \Rn{2}$ compatible with the ball growth lemma.
That is, since estimates will ultimately be added up over balls
$B$, $G$ must have the property that on each $\partial B$, $G =
\tau_{\partial B} \frac{c}{r}$ with $r$ the distance to the center
of $B$.  We will take advantage of the fact that $c$ was an
arbitrary constant; many of the following results are thus valid
with any choice of constants, and it is only much later that we
choose specific values.  Observe already, though, that taking
$c=d_B$ yields an improvement by the $\int \abs{\nabla_{A+G} v}^2$
term to the bounds constructed in Lemma 4.4 of \cite{ss_book}.
Unfortunately, we must choose a more complicated constant $c$ to
make the estimates in Sections \ref{thms} and \ref{l2_inf_of_G}
work.  We now show how to define such a $G$ so that it will be
useful analytically.

Let $\Omega \subset \Rn{2}$ be open and let $\{\mathcal{B}(t)\}_{t \in [0,s]}$ be a family of collections of closed, disjoint balls grown via the ball growth lemma from an initial collection $\mathcal{B}_0$ that covers the set on which $u$ is near $0$.  Let $\mathcal{G}$ denote the subcollection of balls in $\mathcal{B}(s)$ entirely contained in $\Omega$, and let $\mathcal{G}(t)$ denote the balls in $\mathcal{B}(t)$ that are contained in a ball from $\mathcal{G}$, i.e. that remain inside $\Omega$ for all $t$.  For each ball $B \in \mathcal{G}(t)$ we define several quantities.  Let $\tau_{\partial B}:\partial B \rightarrow \Rn{2}$ denote the oriented unit tangent vector field to $\partial B$, and let $a_B$ denote the center of $B$.  Let $d_B = \deg(u/\abs{u},\partial B)$; this is well-defined since the set on which $u$ vanishes is contained in $\mathcal{B}_0$.  Let $\beta_B$ denote a constant, to be specified later, with the property that if $B_1\in \mathcal{G}(t_1)$, $B_2 \in \mathcal{G}(t_2)$, and $B_2 = e^{t_2-t_1} B_1$ (i.e. $B_2$ is grown from $B_1$ without any mergings) then $\beta_{B_1}=\beta_{B_2}$.  In other words, the $\beta_B$ are constant over each annulus produced by the ball construction.  Let $T \subset [0,s]$ denote the finite set of times from the ball growth lemma at which a merging occurs in the growth of $\mathcal{G}(t)$.  We then define the function $G:\Omega \rightarrow \Rn{2}$ by
\begin{equation}\label{G_def}
G(x) = \begin{cases}
        \tau_{\partial B}(x) \frac{d_B \beta_B}{\abs{x-a_B}} &  \text{if} \ x \in \partial B \text{ for some } B\in \mathcal{G}(t), t\in[0,s]\backslash T \\
        0 & \text{otherwise}.
       \end{cases}
\end{equation}
The ball growth lemma guarantees that if $x \in \partial B$ for some $B\in \mathcal{G}(t), t\in[0,s]\backslash T$, then that $t$ is unique, and so $G(x)$ is well defined.  By construction, $G=0$ in $\cup_{B \in \mathcal{G}(0)}B$, and so we can use the above definition of $G$ to extend any function previously defined on $\cup_{B \in \mathcal{G}(0)}B$.  We will frequently do so.

Figure \ref{fig:picture} shows a simple example of balls grown near the boundary of $\Omega$.  Four initial balls, colored light gray, are grown into three final balls, labeled $B_1,B_2,B_3$.  The initial balls are first grown with by a conformal factor of $\tau = \log{2}$ until a merging in required in the balls that become $B_1$.  The result of this merging is the white ball contained in $B_1$.  The growth is then continued with a conformal factor of $\tau = \log(6/5)$ to produce the final balls.  The annuli on which $G$ is defined are colored in dark gray and black.  Since $B_3$ leaves the domain, $G$ is set to zero on the annuli inside it.  $G$ also vanishes on the white region contained in $B_1$.

\begin{figure}
\includegraphics[width=5 in]{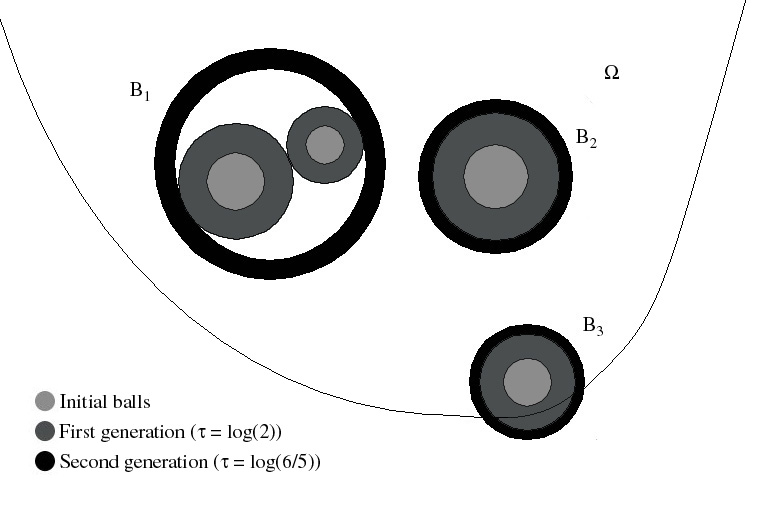}
\caption{Balls grown near the boundary of $\Omega$}
\label{fig:picture}
\end{figure}

With $G$ now properly defined we can show how to couple Lemma \ref{low_boundary} to the ball growth lemma to produce lower bounds on annuli.

\begin{prop}\label{low_growth}
Let $\mathcal{B}_0$ be a finite, disjoint collection of closed balls and let $\Omega \subseteq \Rn{2}$ be open.  Let $\omega = \cup_{B\in \mathcal{B}_0} B$ and denote the collection of balls obtained from $\mathcal{B}_0$ via the ball growth lemma by $\{\mathcal{B}(t)\}$,  $t\ge 0$.  Suppose that $v:\Omega \backslash \omega \rightarrow \mathbb{S}^1$ and $A:\Omega \rightarrow \Rn{2}$ are both $C^1$, and let $G:\Omega \rightarrow \Rn{2}$ be the function defined by \eqref{G_def}.  Fix $s>0$ such that $r(\mathcal{B}(s))\le 1$.  Then, for any $\bar{B}\in \mathcal{B}(s)$ such that $\bar{B} \subset \Omega$, and any $\lambda >0$, we have
\begin{multline}\label{l_g_1}
\frac{1}{2} \int_{\bar{B}\backslash \omega} \abs{\nabla_A v}^2 + \frac{r(\bar{B}) \lambda}{2}\int_{\bar{B}} (\curl{A})^2 - \sum_{B\in \bar{B} \cap \mathcal{B}_0} \frac{r(B)\lambda}{2} \int_{B} (\curl{A})^2
\\ \ge  \frac{1}{2}\int_{\bar{B}\backslash \omega} \abs{\nabla_{A+G} v}^2 + \int_0^s \sum_{B \in \bar{B}\cap \mathcal{B}(t)} \pi d_{B}^2\left(2\beta_{B}-\beta_{B}^2-\frac{\beta_{B}^2r(\mathcal{B}(t))}{2\lambda} \right)dt,
\end{multline}
where we have written $d_B = \deg(u/\abs{u},\partial B)$.
\end{prop}
\begin{proof}
In order to utilize Lemma \ref{function_balls} we define the function
\begin{equation}\label{l_g_2}
\mathcal{F}(x,r) = \frac{1}{2} \int_{B(x,r)} \abs{\nabla_A v}^2 + \frac{r \lambda}{2}\int_{B(x,r)} (\curl{A})^2.
\end{equation}
Differentiating and using \eqref{l_b_1} with $c = \beta_B d_B$,  we arrive at the bound
\begin{equation}\label{l_g_3}
\begin{split}
\frac{\partial \mathcal{F}}{\partial r} & \ge \frac{1}{2} \int_{\partial B(x,r)} \abs{\nabla_A v}^2 + \frac{\lambda}{2}\int_{B(x,r)} (\curl{A})^2\\
& \ge \frac{1}{2} \int_{\partial B(x,r)} \abs{\nabla_{A+G} v}^2 + \frac{\pi d_B^2}{r}(2\beta_B-\beta_B^2) -\frac{\pi d_B^2\beta_B^2}{2\lambda}.
\end{split}
\end{equation}
We now recall the notation of Lemma \ref{function_balls}:
$0 < s_1 < \dotsb < s_K \le s$ denote the times at which merging occurs in the growth of $\mathcal{B}_0$ to $\mathcal{B}(s)$ via the ball growth lemma, and
\begin{equation}
\mathcal{F}(\bar{B} \cap \mathcal{B}(s_k))^{-}=\lim\limits_{t\rightarrow s_k^-} \mathcal{F}(\bar{B} \cap \mathcal{B}(t)).
\end{equation}
By discarding the terms involving $\curl{A}$, we see that
\begin{multline}\label{l_g_4}
\sum_{k=1}^{K} \mathcal{F}(\bar{B}\cap\mathcal{B}(s_k)) - \mathcal{F}(\bar{B} \cap \mathcal{B}(s_k))^{-}  \\ \ge  \sum_{k=1}^K \left(\sum_{B \in \bar{B} \cap \mathcal{B}(s_k)} \frac{1}{2} \int_{B} \abs{\nabla_A v}^2 - \lim_{t\rightarrow s_k^-} \sum_{B \in \bar{B} \cap \mathcal{B}(t)} \frac{1}{2} \int_{B} \abs{\nabla_A v}^2
\right),
\end{multline}
which corresponds to the integral of $\frac{1}{2}\abs{\nabla_{A+G} v}^2$ over the non-annular parts of $\bar{B}\backslash \omega$ since $G=0$ there.  Since the ball growth lemma makes
\begin{equation*}
\frac{d}{dt}r(\mathcal{B}(t)) = r(\mathcal{B}(t)),
\end{equation*}
the expression
\begin{equation*}
\int_0^s \sum_{B \in \bar{B}\cap \mathcal{B}(t)} \frac{r(B)}{2} \int_{\partial B} \abs{\nabla_{A+G} v}^2\, dt
\end{equation*}
corresponds to the integral of $\frac{1}{2}\abs{\nabla_{A+G} v}^2$ over the annular parts of $\bar{B} \backslash \omega$.  We now combine this observation, inequalities \eqref{l_g_3} and \eqref{l_g_4}, and equality \eqref{fn_b_2} to conclude that
\begin{equation}\label{l_g_5}
\begin{split}
&\mathcal{F}(\bar{B}) - \mathcal{F}(\bar{B}\cap \mathcal{B}_0) \\
&\ge \frac{1}{2} \int_{\bar{B} \backslash \omega} \abs{\nabla_{A+G} v}^2 + \int_0^s \sum_{B \in \bar{B}\cap \mathcal{B}(t)} \pi d_{B}^2\left(2\beta_{B}-\beta_{B}^2-\frac{\beta_{B}^2r(B)}{2 \lambda} \right)dt \\
& \ge \frac{1}{2} \int_{\bar{B}\backslash \omega} \abs{\nabla_{A+G}v}^2 + \int_0^s\sum_{B \in \bar{B}\cap \mathcal{B}(t)} \pi d_{B}^2\left(2\beta_{B}-\beta_{B}^2-\frac{\beta_{B}^2r(\mathcal{B}(t))}{2\lambda} \right)dt.
\end{split}
\end{equation}
This is \eqref{l_g_1}.
\end{proof}

The following corollary shows that our method, using $G$, can be used to recover the same estimates found in Proposition 4.3 of \cite{ss_book}.

\begin{cor}\label{low_growth_old}
Under the same assumptions as in Proposition \ref{low_growth} we have
\begin{equation}\label{l_g_o_1}
\frac{1}{2} \int_{\bar{B}\backslash \omega} \abs{\nabla_A v}^2 +
\frac{r(\bar{B})(r_1-r_0)}{2}\int_{\bar{B}} (\curl{A})^2 \ge
\int_0^s \sum_{B \in \bar{B}\cap \mathcal{B}(t)} \pi
d_{B}^2\left(1-\frac{r(\mathcal{B}(t))}{2(r_1-r_0)} \right)dt,
\end{equation}
and
\begin{equation}\label{l_g_o_2}
\frac{1}{2} \int_{\bar{B}\backslash \omega} \abs{\nabla_A v}^2 + \frac{r(\bar{B})(r_1-r_0)}{2}\int_{\bar{B}} (\curl{A})^2  \ge  \pi \abs{d_{\bar{B}}}\left( \log{\frac{r_1}{r_0}} - \log{2}\right),
\end{equation}
where $r_0 := r(\mathcal{B}_0)$ and $r_1 := r(\mathcal{B}(s)) = e^s r_0$.
\end{cor}
\begin{proof}
Set $\lambda = r_1 - r_0$, each $\beta_B = 1$, and disregard the $\abs{\nabla_{A+G} v}$ term and the curl terms on $\mathcal{B}_0$ in \eqref{l_g_1} to get \eqref{l_g_o_1}.  If $\log{\frac{r_1}{r_0}} < \log{2}$, then \eqref{l_g_o_2} follows trivially.  On the other hand, if $\log{\frac{r_1}{r_0}} \ge \log{2}$, then $r_1 \ge 2 r_0$, which implies
\begin{equation}\label{l_g_o_4}
1-\frac{r(\mathcal{B}(t))}{2(r_1-r_0)} \ge 1-\frac{r_1}{2(r_1-r_0)} = \frac{r_1-2r_0}{2(r_1-r_0)} \ge 0.
\end{equation}
Then \eqref{l_g_o_2} follows by noting that $r_1=e^s r_0$,
\begin{equation}\label{l_g_o_3}
\frac{d}{dt}r(\mathcal{B}(t)) = r(\mathcal{B}(t)),
\end{equation}
and (see Lemma 4.2 in \cite{ss_book})
\begin{equation}\label{l_g_o_5}
\sum_{B \in \bar{B}\cap \mathcal{B}(t)} d_{B}^2 \ge \sum_{B \in \bar{B}\cap \mathcal{B}(t)} \abs{d_{B}} \ge \abs{d_{\bar{B}}}.
\end{equation}
\end{proof}

We will need the following modification of the previous corollary later.  It is a slight modification of Proposition 4.3 from \cite{ss_book}.

\begin{lem}\label{alt_low_growth_old}
Under the same assumptions as in Proposition \ref{low_growth} we have
\begin{equation}
\frac{1}{2} \int_{\bar{B}\backslash \omega} \abs{\nabla_A v}^2 + \frac{r(\bar{B})r_1}{2}\int_{\bar{B}} (\curl{A})^2 \\ \ge  \frac{2\pi}{3} \int_0^s \sum_{B \in \bar{B}\cap \mathcal{B}(t)}  d_{B}^2 \;dt.
\end{equation}
\end{lem}
\begin{proof}
 Lemma 4.4 from \cite{ss_book} provides the lower bound on circles, $\partial B= \partial B(a,r)$:
\begin{equation}
\frac{1}{2} \int_{\partial B} \abs{\nabla_A v}^2 + \frac{\lambda}{2} \int_B (\curl{A})^2 \ge \pi \frac{d_B^2}{r} \left( \frac{2\lambda}{2\lambda + r}\right).
\end{equation}
We now set $\lambda = r_1$, bound
\begin{equation*}
 \frac{2r_1}{2r_1 + r} \ge \frac{2}{3},
\end{equation*}
and proceed as before to conclude.

\end{proof}

\section{Initial and final balls}

In this section we record the energy estimates that couple to the
ball construction.  For technical reasons that will arise in the
proof of Theorem \ref{energy_bound} we must use the ball growth
lemma in two phases, just as in Chapter 4 of \cite{ss_book}.  The
first phase produces a collection of initial balls that cover the
set where $\abs{u}$ is far from unity and on which lower bounds of
a type needed in the proof of Theorem \ref{energy_bound} are
satisfied.  This initial collection contains as a subset a
collection of balls on which we initially define the function $G$.
The second phase produces a collection of final balls, grown from
the initial balls, of a chosen size and on which nice lower bounds
hold.  In the final section we finally specify the values of the
$\beta_B$ used to define $G$ and show that certain lower bounds
hold with this choice of constants.

\subsection{The initial balls}

Before we can produce the collection of initial balls, we must first produce a collection of balls that covers the set where $\abs{u}$ is far from unity.  This is accomplished via the following lemma (Proposition 4.8 from \cite{ss_book}), which shows how the radius of this set is controlled by the energy of $\abs{u}$.

\begin{lem}\label{init_set}
Let $M,\varepsilon,\delta>0$ be such that $\varepsilon,\delta<1$, and let $u \in C^1(\Omega,\mathbb{C})$ satisfy the bound $F_{\varepsilon}(\abs{u},\Omega) \le M$.  Then
\begin{equation}
r(\{x\in \Omega_{\varepsilon} \; | \; \abs{u(x)-1} \ge \delta \}) \le C\frac{\varepsilon M}{\delta^2}
\end{equation}
where C is a universal constant and $\Omega_{\varepsilon} = \{ x\in \Omega \;|\; d(x,\partial \Omega) > \varepsilon\}$.
\end{lem}

The next technical result shows how to bound from below the modified radius of sub- and super-level sets.

\begin{lem}\label{radius_bounds}  Let $\Omega\subset \Rn{2}$ be open,  $\Omega_{\varepsilon} = \{ x\in \Omega \;|\; d(x,\partial \Omega) > \varepsilon\}$,  and suppose $\mathcal{B}$ is a finite collection of disjoint, closed balls that cover the set
\begin{equation*}
\{x\in \Omega_{\varepsilon} \; \vert \; \abs{u(x)-1} \ge \delta \}.
\end{equation*}
Let $\mathcal{B}_b$ denote the subcollection of balls in $\mathcal{B}$ that intersect $\partial \Omega_{\varepsilon}$, and let $\mathcal{B}_i$ denote the subcollection of balls in $\mathcal{B}$ contained in the interior of $\Omega_{\varepsilon}$ (i.e. $\mathcal{B} = \mathcal{B}_b \cup \mathcal{B}_i$).  Define $\tilde{\Omega}=\Omega_{\varepsilon} \backslash (\cup_{B\in \mathcal{B}_b} B)$.  For $0<s \le t $ define the sets $\omega_t = \{ x \in \Omega_{\varepsilon} \;|\; \abs{u} \le t\}$, $\omega^t = \{ x \in \Omega_{\varepsilon} \;|\; \abs{u} \ge t\},$ and $\omega_s^t = \omega_s \cup \omega^t$.  Then
\begin{equation}\label{r_b_1}
\begin{split}
& r_{\Omega_{\varepsilon}}(\omega_t) \ge r(\omega_t \cap \tilde{\Omega}) \text{ for } t\in(0,1-\delta), \\
& r_{\Omega_{\varepsilon}}(\omega^t) \ge r(\omega^t \cap \tilde{\Omega}) \text{ for } t\in(1+\delta,\infty), \text{ and} \\
& r_{\Omega_{\varepsilon}}(\omega_s^t) \ge r(\omega_s^t \cap \tilde{\Omega}) \text{ for } s\in(0,1-\delta), t\in (1+\delta,\infty).
\end{split}
\end{equation}
\end{lem}
\begin{proof}
Suppose that $t\in(0,1-\delta)$ and let Int$(\cdot)$ denote the interior of a set.  Write $V =\cup_{B\in \mathcal{B}}B$ and $V_i =\cup_{B\in \mathcal{B}_i}B$.  Since the inclusions
\begin{equation}\label{r_b_2}
\text{Int}(V) \supseteq \text{Int}(\{x\in \Omega_{\varepsilon} \; | \; \abs{u(x)-1} \ge \delta \}) \supset \omega_t
\end{equation}
hold, we have that $\omega_t \cap \tilde{\Omega} = \omega_t \cap V_i$, and hence $r(\omega_t \cap \tilde{\Omega})=r(\omega_t \cap V_i)$.  When combined with the fact that $V_i$ is a compact subset of $\Omega_{\varepsilon}$ and $\partial V_i \cap \omega_t = \varnothing$, this yields the first estimate in \eqref{r_b_1}.  Similar arguments prove the second and third assertions.
\end{proof}

We now construct the initial balls.  The following proposition is the analogue of Proposition 4.7 of \cite{ss_book}, but here we have an extra term of the form
\begin{equation*}
\int \abs{\nabla_{A+G} v}^2.
\end{equation*}
Note that items 1, 2, and 3 are the same as those found in \cite{ss_book}; item 4 is new.

\begin{prop}\label{init_balls}
Let $\alpha \in (0,1)$.  There exists $\varepsilon_0 >0$ (depending on $\alpha$) such that for $\varepsilon \le \varepsilon_0$ and $u \in C^1(\Omega,\mathbb{C})$ with $F_{\varepsilon}(\abs{u},\Omega) \le \varepsilon^{\alpha-1},$ the following hold.

There exists a finite, disjoint collection of closed balls, denoted by $\mathcal{B}_0$, with the following properties.

1. $r(\mathcal{B}_0) = C \varepsilon^{\alpha/2}$, where $C$ is a universal constant.

2. $\{x\in \Omega_{\varepsilon} \; | \; \abs{u(x)-1} \ge \delta \} \subset V_0 := \Omega_{\varepsilon} \cap \left(\cup_{B\in \mathcal{B}_0} B\right)$, where $\delta = \varepsilon^{\alpha/4}$.

3. Write $v= u/\abs{u}$.  For $t \in (0,1-\delta)$ we have the estimate
\begin{equation}\label{i_b_11}
\frac{1}{2}\int_{V_0\backslash \omega_t} \abs{\nabla_A v}^2 + \frac{r(\mathcal{B}_0)^2}{2}\int_{V_0} (\curl{A})^2 \ge \pi D_0 \left( \log{\frac{r(\mathcal{B}_0)}{r_{\Omega_\varepsilon}(\omega_t)}} - C \right),
\end{equation}
where
\begin{equation}
D_0 = \sum_{\substack{B\in \mathcal{B}_{0} \\ B \subset \Omega_\varepsilon}} \abs{d_B}.
\end{equation}

4. There exists a family of finite collections of closed, disjoint balls $\{\mathcal{C}(s)\}_{s\in[0,\sigma]}$,
all of which are contained in in $V_0$, and that are grown according to the ball growth lemma from an initial collection, $\mathcal{C}(0)$, that covers the set $\omega_{1/2}^{3/2}\cap V_0$.  The number $\sigma$ is such that $r(\mathcal{C}(\sigma)) = \frac{3}{8} r(\mathcal{B}_0)$. Let $G:V_0\rightarrow \Rn{2}$ be the function defined by using $\Omega_{\varepsilon}$ and $\{\mathcal{C}(s)\}_{s\in[0,\sigma]}$ in \eqref{G_def} and then extended by zero to the rest of $V_0$. For each $\lambda >0$ we have the estimate

\begin{multline}\label{i_b_12}
\frac{1}{2} \int_{V_0 \backslash \omega_{1/2}^{3/2}} \abs{\nabla_A v}^2 + \sum_{B \in \mathcal{B}_0} \frac{r(B)\lambda}{2} \int_{B \cap \Omega} (\curl{A})^2 \\
\\
\ge  \int_0^{\sigma} \sum_{\substack{\bar{B} \in \mathcal{C}(\sigma) \\ \bar{B} \subset \Omega_{\varepsilon}}}\sum_{B \in \bar{B} \cap \mathcal{C}(t)} \pi d_B^2\left(2\beta_B -\beta_B^2 - \frac{\beta_B^2 r(\mathcal{C}(t))}{2 \lambda} \right)dt
+\frac{1}{2} \int_{V_0 \backslash \omega_{1/2}^{3/2}}
\abs{\nabla_{A+G} v}^2.
\end{multline}
\end{prop}
\begin{proof}
We break the proof into six steps.  The first four consist of finding four collections of balls that are used to create the initial collection $\mathcal{B}_0$.  The last two steps prove the estimates of items 3 and 4.\\\\
Step 1.

Using $M = \varepsilon^{\alpha -1}$ and $\delta = \varepsilon^{\alpha/4}$ in Lemma \ref{init_set} produces a collection of disjoint, closed balls $\mathcal{E}$ that cover the set $\{x\in \Omega_{\varepsilon} \; | \; \abs{u(x)-1} \ge \delta \}$ such that $R:=r(\mathcal{E}) \le C\varepsilon^{\alpha/2}$.  We will eventually need to use Lemma \ref{radius_bounds}, so we employ its notation by breaking the collection $\mathcal{E}$ into subcollections $\mathcal{E}_i$ and $\mathcal{E}_b$ and defining the set $\tilde{\Omega}=\Omega_{\varepsilon}\backslash (\cup_{B\in \mathcal{E}_b}B$).\\\\
Step 2.

  By the definition of the radius of a set, for any $t\in (0,1-\delta)$ we can cover $\omega_t \cap \tilde{\Omega}$ by a collection of disjoint balls, denoted by $\mathcal{B}_t^0$, with total radius less than $2r(\omega_t \cap \tilde{\Omega})$.  Since $r(\omega_t \cap \tilde{\Omega}) \le R$, we can use Lemma \ref{ball_growth_lemma} to grow the collection $\mathcal{B}_t^0$ into a collection $\mathcal{B}_t$ such that $r(\mathcal{B}_t)=2R$.  We then utilize Corollary \ref{low_growth_old}
on each of the balls in $\mathcal{B}_t$ that is contained in $\tilde{\Omega}$ and sum to get the estimate
\begin{equation}\label{i_b_4}
\frac{1}{2} \int_{V_t \backslash \omega_t} \abs{\nabla_A v}^2 + \frac{4R^2}{2} \int_{V_t} (\curl{A})^2 \ge \pi D_t \left(\log{\frac{2R}{2r(\omega_t \cap \tilde{\Omega})}} -\log{2}  \right),
\end{equation}
where
\begin{equation*}
\begin{split}
& V_t = \tilde{\Omega} \cap \left(\cup_{B\in \mathcal{B}_t} B\right), \text{ and} \\
& D_t = \sum_{\substack{B\in \mathcal{B}_t \\ B \subset \tilde{\Omega}}} \abs{d_B}.
\end{split}
\end{equation*}
Choose $\bar{t} \in (0,1-\delta)$ such that $D_{\bar{t}}$ is minimal.\\\\
Step 3.

Let $m$ denote the supremum of
\begin{equation*}
\mathcal{F}(K) := \frac{1}{2} \int_{(K \cap \tilde{\Omega})\backslash \omega} \abs{\nabla_A v}^2 + \frac{4R^2}{2} \int_{K \cap \tilde{\Omega}} (\curl{A})^2
\end{equation*}
over compact $K \subset \Omega$ such that $r(K) < 2R$.  Choose $K$ so that $r(K) < 2R$ and $\mathcal{F}(K) \ge m-1$.  Cover $K$ by a collection of disjoint, closed balls $\mathcal{K}$ such that $r(\mathcal{K})=2R$ (the existence of such a collection is guaranteed by the ball growth lemma).\\\\
Step 4.

We can cover $\omega_{1/2}^{3/2} \cap \tilde{\Omega}$ by a collection of disjoint balls, denoted by $\mathcal{C}_0$, with radius less than $\frac{3}{2} r(\omega_{1/2}^{3/2} \cap \tilde{\Omega})$.  We use the ball growth lemma, applied to $\mathcal{C}_0$, to produce a family of collections $\{\mathcal{C}(s) \}$ with $s\in (0,\sigma)$,
\begin{equation*}
\sigma = \log{\left( \frac{3R}{r(\mathcal{C}_0)} \right)}.
\end{equation*}
Let $\mathcal{C} = \mathcal{C}(\sigma)$ and note that by construction $r(\mathcal{C})=3R$. \\\\
Step 5.

Define $\mathcal{B}_0$ to be a collection of disjoint balls that cover the balls in $\mathcal{B}_{\bar{t}}$, $\mathcal{K}$, $\mathcal{C}$, and $\mathcal{E}$.  We may choose such a collection so that $r(\mathcal{B}_0)=8R$.  Let $V_0 = \Omega_{\varepsilon} \cap \left(\cup_{B\in \mathcal{B}_0} B\right)$.  Then
\begin{equation}\label{i_b_5}
I:= \frac{1}{2}\int_{V_0\backslash \omega_t} \abs{\nabla_A v}^2 + \frac{r(\mathcal{B}_0)^2}{2}\int_{V_0} (\curl{A})^2 \ge \mathcal{F}(K) + \frac{1}{2}\int_{\omega \backslash \omega_t} \abs{\nabla_A v}^2,
\end{equation}
and by the construction of $K$ and $V_t$ for any $t\in(0,1-\delta)$, this implies
\begin{equation}\label{i_b_6}
\begin{split}
I+1 &\ge \mathcal{F}(V_t) + \frac{1}{2}\int_{\omega \backslash \omega_t} \abs{\nabla_A v}^2 \\
& \ge \frac{1}{2}\int_{V_t\backslash \omega_t} \abs{\nabla_A v}^2 + \frac{4R^2}{2}\int_{V_t} (\curl{A})^2 \\\\
& \ge \pi D_t \left( \log{\frac{2R}{2r(\omega_t \cap \tilde{\Omega})}} -\log{2}\right) \\
& \ge \pi D_t \left( \log{\frac{r(\mathcal{B}_0)}{r_{\Omega_\varepsilon}(\omega_t)}} - C \right),
\end{split}
\end{equation}
where the last line follows from \eqref{r_b_1} and the fact that $r(\mathcal{B}_0)=8R$.  By the choice of $\bar{t}$,
\begin{equation}\label{i_b_7}
D_t \ge D_{\bar{t}} = \sum_{\substack{B\in \mathcal{B}_{\bar{t}} \\ B \subset \tilde{\Omega}}} \abs{d_B}.
\end{equation}
We break the collection of balls in the last sum in \eqref{i_b_7} into two subcollections:
\begin{equation*}
\begin{split}
& I_1:=\{ B \in \mathcal{B}_{\bar{t}} \;|\; B \subseteq \tilde{\Omega}, B\subseteq B'\in \mathcal{B}_0 \text{ so that } B' \cap \partial \Omega_{\varepsilon} \neq \varnothing \}\\
& I_2:=\{ B \in \mathcal{B}_{\bar{t}} \;|\; B \subseteq \tilde{\Omega}, B\subseteq B'\in \mathcal{B}_0 \text{ so that } B' \subseteq \Omega_{\varepsilon}  \}.
\end{split}
\end{equation*}
Then
\begin{equation}\label{i_b_13}
\sum_{\substack{B\in \mathcal{B}_{\bar{t}} \\ B \subset \tilde{\Omega}}} \abs{d_B}
= \sum_{B \in I_1}\abs{d_B} + \sum_{B \in I_2}\abs{d_B}
\ge 0 + \sum_{\substack{B\in \mathcal{B}_{0} \\ B \subset \Omega_\varepsilon}}\abs{d_B} = D_0,
\end{equation}
where the inequality follows from Lemma 4.2 in \cite{ss_book}.
Combining \eqref{i_b_6}, \eqref{i_b_7}, and \eqref{i_b_13} yields \eqref{i_b_11}.\\\\
Step 6.

Let $U$ be the union of the balls in $\mathcal{C}_0$ that are contained in $\Omega_{\varepsilon}$ and $W$ be the union of the balls in $\mathcal{C}$ that are contained in $\Omega_{\varepsilon}$.  Then applying Proposition \ref{low_growth} to each $\bar{B} \in \mathcal{C}$ such that $\bar{B} \subset \Omega_{\varepsilon}$ and summing, we get the estimate
\begin{multline}\label{i_b_8}
\frac{1}{2} \int_{W \backslash U} \abs{\nabla_A v}^2 + \sum_{\substack{\bar{B} \in \mathcal{C} \\ \bar{B} \subset \Omega_{\varepsilon}}}\frac{r(\bar{B})\lambda}{2} \int_{\bar{B}} (\curl{A})^2  \\ \ge \frac{1}{2} \int_{W \backslash U} \abs{\nabla_{A+G} v}^2 + \int_0^{\sigma} \sum_{\substack{\bar{B} \in \mathcal{C} \\ \bar{B} \subset \Omega_{\varepsilon}}}\sum_{B \in \bar{B} \cap \mathcal{C}(t)} \pi d_B^2\left(2\beta_B -\beta_B^2 - \frac{\beta_B^2 r(\mathcal{C}(t))}{2 \lambda}  \right)dt.
\end{multline}
$G$ vanishes in the regions $V_0 \backslash W$ and $U \backslash \omega_{1/2}^{3/2}$, so
\begin{equation}\label{i_b_9}
\frac{1}{2} \int_{(V_0 \backslash W)\cup (U \backslash \omega_{1/2}^{3/2})} \abs{\nabla_A v}^2 = \frac{1}{2} \int_{(V_0 \backslash W)\cup (U \backslash \omega_{1/2}^{3/2})} \abs{\nabla_{A+G} v}^2.
\end{equation}
Adding \eqref{i_b_9} to both sides of \eqref{i_b_8} and noting that
\begin{equation}\label{i_b_10}
\sum_{\substack{\bar{B} \in \mathcal{C} \\ \bar{B} \subset \Omega_{\varepsilon}}}\frac{r(\bar{B})\lambda}{2} \int_{\bar{B}} (\curl{A})^2 \le \sum_{B \in \mathcal{B}_0} \frac{r(B)\lambda}{2} \int_{B \cap \Omega} (\curl{A})^2
\end{equation}
yields \eqref{i_b_12}.
\end{proof}

\subsection{The final balls}

The next proposition constructs the final balls from the initial ones constructed in Proposition \ref{init_balls}.  Items 1, 2, and 3 are the same as those of Theorem 4.1 of \cite{ss_book}; item 4 contains the novel estimate with the $G$-term.

\begin{prop}\label{final_balls}
Let $\alpha \in (0,1)$.  There exists $\varepsilon_0 >0$ (depending on $\alpha$) such that for $\varepsilon \le \varepsilon_0$ and $u \in C^1(\Omega,\mathbb{C})$ with $F_{\varepsilon}(\abs{u},\Omega) \le \varepsilon^{\alpha - 1}$, the following hold.

For any $1 > r > C \varepsilon^{\alpha/2}$, where $C$ is a universal constant, there exists a finite, disjoint collection of closed balls, denoted by $\mathcal{B}$, with the following properties.

1. $r(\mathcal{B})=r$.

2. $\{x\in \Omega_{\varepsilon} \; | \; \abs{u(x)-1} \ge \delta \} \subset V := \Omega_{\varepsilon} \cap \left(\cup_{B\in \mathcal{B}} B\right)$, where $\delta = \varepsilon^{\alpha/4}$.

3. Write $v= u/\abs{u}$.  For $t \in (0,1-\delta)$ we have the estimate
\begin{equation}\label{f_b_1}
\frac{1}{2}\int_{V\backslash \omega_t} \abs{\nabla_A v}^2 + \frac{r^2}{2}\int_{V} (\curl{A})^2 \ge \pi D \left( \log{\frac{r}{r_{\Omega_\varepsilon}(\omega_t)}} - C \right),
\end{equation}
where
\begin{equation}
D = \sum_{\substack{B\in \mathcal{B} \\ B \subset \Omega_\varepsilon}} \abs{d_B}.
\end{equation}

4. Let $G:\Omega \rightarrow \Rn{2}$ be the extension, according to \eqref{G_def}, of the $G$ from item 4 in Proposition \ref{init_balls}.  Write $s = \log{\frac{r}{r(\mathcal{B}_0)}}$.
Then
\begin{multline}\label{f_b_2}
\frac{1}{2} \int_{V \backslash \omega_{1/2}^{3/2}} \abs{\nabla_A v}^2 + \sum_{\bar{B} \in \mathcal{B}}\frac{r(\bar{B}) (r-r(\mathcal{B}_0))}{2}\int_{\bar{B} \cap \Omega} (\curl{A})^2 \ge \frac{1}{2} \int_{V \backslash \omega_{1/2}^{3/2}} \abs{\nabla_{A+G} v}^2 \\ + \int_0^s \sum_{\substack{\bar{B} \in \mathcal{B} \\ \bar{B} \subset \Omega_{\varepsilon}}} \sum_{B \in \bar{B} \cap \mathcal{B}(t)} \pi d_B^2 \left(2\beta_B -\beta_B^2 - \frac{\beta_B^2 r(\mathcal{B}(t))}{2(r-r(\mathcal{B}_0))} \right)dt \\ + \int_0^{\sigma} \sum_{\substack{\bar{B} \in \mathcal{C}(\sigma) \\ \bar{B} \subset \Omega_{\varepsilon}}}\sum_{B \in \bar{B} \cap \mathcal{C}(t)} \pi d_B^2\left(2\beta_B -\beta_B^2 - \frac{\beta_B^2 r(\mathcal{C}(t))}{2 (r(\mathcal{C}(\sigma)) - r(\mathcal{C}_0))} \right)dt.
\end{multline}

\end{prop}

\begin{proof}
Lemma \ref{init_balls} provides an initial set of disjoint, closed balls $\mathcal{B}_0$.  We grow these according to the ball growth lemma to produce $\{\mathcal{B}(t)\}_{t\in[0,s]}$ with $s$ chosen so that $r(\mathcal{B}(s))=r$, i.e. $s = \log{\frac{r}{r(\mathcal{B}_0)}}$.  By construction, items 1 and 2 are proved.  Let $\mathcal{B} = \mathcal{B}(s)$, and write $V = \Omega_{\varepsilon} \cap \cup_{B \in \mathcal{B}}B$, $V_0 = \Omega_{\varepsilon} \cap \cup_{B \in \mathcal{B}_0}B$.  Let $G:V_0 \rightarrow \Rn{2}$ be the function defined in item 4 of Proposition \ref{init_balls}.  We then use $\mathcal{B}_0$ and $\mathcal{B}$ to extend $G:\Omega \rightarrow \Rn{2}$ according to \eqref{G_def}.

We analyze the balls in $\mathcal{B}$ according to whether or not they are contained entirely in $\Omega_{\varepsilon}$.  For balls $\bar{B} \in \mathcal{B}$ such that $\bar{B} \subset \Omega_{\varepsilon}$, we use \eqref{l_g_o_2}, and for the other balls we use the trivial non-negative bound.  Summing over all balls in $\mathcal{B}$, we get
\begin{equation}\label{f_b_13}
\frac{1}{2} \int_{V\backslash V_0} \abs{\nabla_A v}^2 +
\sum_{\bar{B}\in
\mathcal{B}}\frac{r(\bar{B})(r-r(\mathcal{B}_0))}{2}
\int_{\bar{B}\cap \Omega} (\curl{A})^2  \ge \pi D \left(
\log{\frac{r}{r(\mathcal{B}_0)}}-\log{2}   \right).
\end{equation}
Adding \eqref{i_b_11} to \eqref{f_b_13} and noting that $D_0 \ge D$ then yields \eqref{f_b_1}.

To prove \eqref{f_b_2} we proceed similarly, using different estimates for the balls in $\mathcal{B}$ according to whether or not they are contained in $\Omega_{\varepsilon}$.  For balls $\bar{B} \in \mathcal{B}$ such that $\bar{B} \subset \Omega_{\varepsilon}$ we use Proposition \ref{low_growth} to get the estimate
\begin{multline}\label{f_b_3}
\frac{1}{2} \int_{\bar{B} \backslash V_0} \abs{\nabla_A v}^2 + \frac{r(\bar{B}) \lambda}{2}\int_{\bar{B}} (\curl{A})^2 - \sum_{B \in \bar{B} \cap \mathcal{B}_0}\frac{r(B) \lambda}{2}\int_{B} (\curl{A})^2 \\
\ge \frac{1}{2} \int_{\bar{B} \backslash V_0} \abs{\nabla_{A+G} v}^2 + \int_0^s \sum_{B\in \bar{B} \cap \mathcal{B}(t)} \pi d_B^2 \left(  2\beta_B -\beta_B^2 - \frac{\beta_B^2 r(\mathcal{B}(t))}{2\lambda}  \right).
\end{multline}
On the other hand, the construction of $G$ guarantees that it vanishes on all balls $\bar{B} \in \mathcal{B}$ such that $\bar{B} \cap \partial \Omega_{\varepsilon}\neq \varnothing$, and so for such $\bar{B}$ we trivially have the estimate
\begin{multline}\label{f_b_4}
\frac{1}{2} \int_{(\bar{B}\cap \Omega) \backslash V_0} \abs{\nabla_A v}^2 + \frac{r(\bar{B}) \lambda}{2}\int_{\bar{B}\cap \Omega} (\curl{A})^2 - \sum_{B \in \bar{B} \cap \mathcal{B}_0}\frac{r(B) \lambda}{2}\int_{B \cap \Omega} (\curl{A})^2   \\
\ge \frac{1}{2} \int_{(\bar{B} \cap \Omega) \backslash V_0} \abs{\nabla_{A+G} v}^2.
\end{multline}
Summing \eqref{f_b_3} and \eqref{f_b_4} over all balls in $\mathcal{B}$ then yields the estimate
\begin{multline}\label{f_b_5}
\frac{1}{2} \int_{V \backslash V_0} \abs{\nabla_A v}^2 + \sum_{\bar{B} \in \mathcal{B}}\frac{r(\bar{B}) \lambda}{2}\int_{\bar{B} \cap \Omega} (\curl{A})^2 - \sum_{B \in \mathcal{B}_0}\frac{r(B) \lambda}{2}\int_{B\cap \Omega} (\curl{A})^2 \\
\ge \frac{1}{2} \int_{V \backslash V_0} \abs{\nabla_{A+G} v}^2 + \int_0^s \sum_{\substack{\bar{B} \in \mathcal{B} \\ \bar{B} \subset \Omega_{\varepsilon}}} \sum_{B \in \bar{B} \cap \mathcal{B}(t)} \pi d_B^2 \left(2\beta_B -\beta_B^2 - \frac{\beta_B^2 r(\mathcal{B}(t))}{2\lambda} \right)dt.
\end{multline}
We insert $\lambda = r-r(\mathcal{B}_0)$ into \eqref{f_b_5} and $\lambda = r(\mathcal{C}(\sigma)) - r(\mathcal{C}_0) =   \frac{3 r(\mathcal{B}_0)}{8} - r(\mathcal{C}_0)$ into \eqref{i_b_12} and add the estimates together.  Noting that
\begin{equation}\label{f_b_6}
\frac{3r(\mathcal{B}_0)}{8} - r(\mathcal{C}_0) -r + r(\mathcal{B}_0) \le C \varepsilon^{\alpha/2}-r \le 0,
\end{equation}
we arrive at the estimate \eqref{f_b_2}.

\end{proof}

\subsection{Degree analysis and selection of the $\beta_B$ values}

We will now select the values of the $\beta_B$ used to define $G$.  Ultimately, later in Theorem \ref{norm_switch}, we will get rid of $G$ altogether by bounding its $L^{2,\infty}$ norm by a term of the order $D^2$.  This bound, the proof of which is Proposition \ref{g_norm}, requires the values of the $\beta_B$ to be small.  However, since they play a role in the lower bounds of Proposition \ref{final_balls}, we can not choose the $\beta_B$ to be too small.  We balance these two demands by introducing a parameter $\eta$ to measure when $\beta_B$ must be small and when it can assume the natural choice for its value, $1$.

The next two results establish that for a ball $\bar{B} \in \mathcal{B}(s)$ there is a transition time (depending on $\eta$) in the family $\bar{B} \cap \mathcal{B}(t)$ before which we can take $\beta_B =1$, and after which we must use something more complicated.

\begin{lem}\label{degree_analysis}
Let $\mathcal{B}_0$ be a finite collection of disjoint, closed balls.  Suppose further that the collection $\mathcal{B}_0$ has the degree covering property that for all balls $B \subset \Omega \backslash (\cup_{S\in \mathcal{B}_0}S)$, it is the case that $d_B =0$.  In other words, the collection $\mathcal{B}_0$ covers all of the vortices.  Let $\mathcal{B}(t)$, $t\in[0,s]$, be a t-parameterized family of finite collections of disjoint, closed balls.  Suppose that $\mathcal{B}_0 =\mathcal{B}(0)$ and that
\begin{equation}\label{d_a_1}
\bigcup_{B \in \mathcal{B}(t_1)} B \subseteq \bigcup_{B \in \mathcal{B}(t_2)}B \text{ for } t_1\le t_2.
\end{equation}
Fix $\bar{B} \in \mathcal{B}(s)$. Define the negative and positive vorticity masses by
\begin{equation}\label{d_a_2}
\begin{split}
& N(t):=\sum_{\substack{B \in \bar{B} \cap \mathcal{B}(t) \\ d_B<0}} \abs{d_B} \\
& P(t):=\sum_{\substack{B \in \bar{B} \cap \mathcal{B}(t) \\ d_B>0}} d_B.
\end{split}
\end{equation}
Then for any $\eta \in (0,1)$, the following hold.

1. If $d_{\bar{B}}\ge0$ and the inequality
\begin{equation} \label{d_a_3}
N(s_0)\le \eta P(s_0)
\end{equation}
holds for some $s_0 \in [0,s]$, then $N(t) \le \eta P(t)$ for all $t\in[s_0,s]$.

2. If $d_{\bar{B}}<0$ and the inequality
\begin{equation} \label{d_a_4}
P(s_0)\le \eta N(s_0)
\end{equation}
holds for some $s_0 \in [0,s]$, then $P(t) \le \eta N(t)$ for all $t\in[s_0,s]$.
\end{lem}

\begin{proof}
Take $d_{\bar{B}}\ge0$; the following proves \eqref{d_a_3}, and a similar argument with $d_{\bar{B}}<0$ proves \eqref{d_a_4}.  Let $n(t) = \#\mathcal{B}(t)$.  Then by the inclusion property \eqref{d_a_1}, $n(t)$ is a decreasing $\mathbb{N}$-valued function.  Hence there exist finitely many times $0 = t_0 < \dotsb < t_K=s$ such that $n(t)$ is constant on $(t_i,t_{i+1})$.  This implies that for $t_i < s < t < t_{i+1}$ and $B\in\mathcal{B}(t)$, there exists exactly one ball $B' \in \mathcal{B}(s)$ such that $B' \subseteq B$, and by the degree covering property, $d_B = d_{B'}$.  It follows that $N(t)$ and $P(t)$ are also constant on each $(t_i,t_{i+1})$.  Then it suffices to show that if $N(t_k) \le \eta P(t_k)$, then $N(t_{k+1}) \le \eta P(t_{k+1})$.

Given a ball $C\in \mathcal{B}(t_{k+1})$, the inclusion property guarantees that there is a finite collection $\{B_1,\dotsc,B_j\} \subseteq \mathcal{B}(t_k)$ such that $B_i \subseteq C$ for $i=1,\dotsc,j$.  We then get
\begin{equation}\label{d_a_5}
\begin{split}
& \abs{d_C} = -\sum_{\substack{i\in \{1,\dotsc,j\} \\d_{B_i}\ge 0}} d_{B_i} + \sum_{\substack{i\in \{1,\dotsc,j\} \\d_{B_i}<0}} \abs{d_{B_i}}  \text{   if }d_C<0, \text{ and}  \\
& \abs{d_C} = \sum_{\substack{i\in \{1,\dotsc,j\} \\d_{B_i}\ge 0}} d_{B_i} - \sum_{\substack{i\in \{1,\dotsc,j\} \\d_{B_i}<0}} \abs{d_{B_i}}  \text{   if }d_C\ge 0.
\end{split}
\end{equation}
We must now subdivide the collection $\bar{B} \cap \mathcal{B}(t_k)$ according to the degrees of balls in $\bar{B} \cap \mathcal{B}(t_{k+1})$.  Define the collections
\begin{equation*}
\begin{split}
&I_{-,-} = \{ B \in \bar{B} \cap \mathcal{B}(t_k) \; | \; d_B <0,  \exists B'\in \bar{B}\cap \mathcal{B}(t_{k+1})\text{ s.t. } B\subset B', d_{B'}< 0 \} \\
&I_{-,+} = \{ B \in \bar{B} \cap \mathcal{B}(t_k) \; | \; d_B <0,  \exists B'\in \bar{B}\cap \mathcal{B}(t_{k+1})\text{ s.t. } B\subset B', d_{B'}\ge 0 \} \\
&I_{+,-} = \{ B \in \bar{B} \cap \mathcal{B}(t_k) \; | \; d_B \ge 0,  \exists B'\in \bar{B}\cap \mathcal{B}(t_{k+1})\text{ s.t. } B\subset B', d_{B'} <  0 \} \\
&I_{+,+} = \{ B \in \bar{B} \cap \mathcal{B}(t_k) \; | \; d_B \ge 0,  \exists B'\in \bar{B}\cap \mathcal{B}(t_{k+1})\text{ s.t. } B\subset B', d_{B'} \ge 0 \}.
\end{split}
\end{equation*}
Now we can estimate
\begin{multline} \label{d_a_6}
\eta \sum_{B \in I_{-,+}} \abs{d_B} + \sum_{B \in I_{-,-}} \abs{d_B} \le \sum_{B \in I_{-,+}}\abs{d_B} + \sum_{B \in I_{-,-}} \abs{d_B} = N(t_k) \\
\le \eta P(t_k) = \eta \sum_{B \in I_{+,-}} d_B + \eta \sum_{B \in I_{+,+}}d_B \le \sum_{B \in I_{+,-}} d_B + \eta \sum_{B \in I_{+,+}}d_B.
\end{multline}
After regrouping terms according to containment and using \eqref{d_a_5} and \eqref{d_a_6} we conclude
\begin{equation}\label{d_a_7}
N(t_{k+1}) = \sum_{B \in I_{-,-}} \abs{d_B} - \sum_{B \in I_{+,-}}d_B \le \eta \sum_{B \in I_{+,+}} d_B - \eta \sum_{B \in I_{-,+}}\abs{d_B} = \eta P(t_{k+1}).
\end{equation}
\end{proof}

We use this lemma to define the transition times.

\begin{cor}\label{degree_analysis_firstgen}
Assume the hypotheses and notation of Lemma \ref{degree_analysis}.  If $d_{\bar{B}} \ge 0$ then there exists $t_0 \in [0,s]$ such that $\eta P(t) <  N(t)$ for $t\in[0,t_0)$ and $N(t) \le \eta P(t)$ for $t\in[t_0,s]$.  Similarly, if $d_{\bar{B}} < 0$ then there exists $t_0 \in [0,s]$ such that $\eta N(t) <  P(t)$ for $t\in[0,t_0)$ and $P(t) \le \eta N(t)$ for $t\in[t_0,s]$.  We call these times, $t_0$, the transition times.
\end{cor}
\begin{proof}
Assume $d_{\bar{B}} \ge 0$.  Since there is only one ball in $\bar{B} \cap \mathcal{B}(s)$, and the degree in $\bar{B}$ is nonnegative, the inequality $N(s) \le \eta P(s)$ is satisfied trivially.  An application of Lemma \ref{degree_analysis} proves the existence of $t_0$.  A similar argument works for the case when $d_{\bar{B}} <0$.
\end{proof}

With the transition times defined we can finally set the values of the $\beta_B$.  Define the collection $\{\mathcal{D}(t) \}_{t\in[0,s+\sigma]}$ by
\begin{equation}\label{f_b_14}
\mathcal{D}(t) = \begin{cases}
            \mathcal{C}(t), & t\in[0,\sigma) \\
            \mathcal{B}(t-\sigma), & t\in [\sigma,s+\sigma].
                 \end{cases}
\end{equation}
Let $\eta \in (0,1)$.  For each $\bar{B}\in\mathcal{B}$ let $t_{\bar{B}}\in [0,s+\sigma]$ denote the transition time for the collection $\bar{B} \cap \mathcal{D}(t)$ obtained from Corollary \ref{degree_analysis_firstgen} (the times depend on $\eta$).  We now specify the values of $\beta_B$ in the definition of $G$.  Note that the construction of $G$ only requires specifying the values of $\beta_B$ for those balls $B$ such that $B \subset \bar{B} \in \mathcal{B}$ with $\bar{B} \subset \Omega_{\varepsilon}$.  Then for  $B \in \bar{B} \cap \mathcal{D}(t)$ for some $\bar{B} \in \mathcal{B}$, we define
\begin{equation}\label{f_b_15}
\beta_B = \begin{cases}
        1, & \text{ if } t\in[0,t_{\bar{B}}) \\
        \abs{d_{\bar{B}}}^{\frac{1}{2}} \left( \sum\limits_{B' \in \bar{B} \cap \mathcal{D}(t)}d_{B'}^2         \right)^{-\frac{1}{2}}, & \text{ if } t\in[t_{\bar{B}},s+\sigma].
          \end{cases}
\end{equation}
Note that if
\begin{equation*}
\sum\limits_{B' \in \bar{B} \cap \mathcal{D}(t)}d_{B'}^2 =0,
\end{equation*}
then $d_{\bar{B}} = 0$ as well, and we take the second case in \eqref{f_b_15} to equal $0$.  Further, note that in the second case, the $\beta_B$ are chosen so that for $t \in [t_{\bar{B}},s+\sigma]$
\begin{equation}
\sum_{B\in \bar{B}\cap \mathcal{D}(t)} d_B^2 \beta_B^2 = \abs{d_{\bar{B}}}.
\end{equation}

The following proposition shows that $G$ is still useful for the lower bounds with these values of $\beta_B$.

\begin{prop}\label{final_balls_part_2}
With $G$ defined as above, and under the assumptions of Proposition \ref{final_balls}, we have the estimate
\begin{equation}\label{f_b_12}
\frac{1}{2} \int_{V \backslash \omega_{1/2}^{3/2}} \abs{\nabla_A
v}^2 + \frac{r^2}{2}\int_V (\curl{A})^2  \ge \frac{1}{2}\int_{V
\backslash \omega_{1/2}^{3/2}} \abs{\nabla_{A+G} v}^2 + \pi D
\left(\log{\frac{r}{r_{\Omega_{\varepsilon}}(\omega_{1/2}^{3/2})}}
- C \right).
\end{equation}
\end{prop}

\begin{proof}
To prove \eqref{f_b_12} we must deal with the sums in the integrands in \eqref{f_b_2}.  We begin by showing that the terms in parentheses are nonnegative.  Since $r(\mathcal{B}_0) = C\varepsilon^{\alpha/2}$ and $\beta_B \le 1$, we can estimate
\begin{equation}
\begin{split}
&2\beta_B -\beta_B^2 - \frac{\beta_B^2 r(\mathcal{B}(t))}{2(r-r(\mathcal{B}_0))}
= \beta_B^2\left(\frac{2}{\beta_B} -1-\frac{r(\mathcal{B}(t))}{2(r-r(\mathcal{B}_0))}   \right) \\
&\ge \beta_B^2\left(  1 - \frac{r(\mathcal{B}(t))}{2(r-r(\mathcal{B}_0))}    \right)
\ge \beta_B^2 \left(\frac{r-2r(\mathcal{B}_0)}{2(r-r(\mathcal{B}_0))}   \right) \ge 0.
\end{split}
\end{equation}
By construction,
\begin{equation}\label{f_b_22}
r(\mathcal{C}_0) < \frac{3}{2} r(\omega_{1/2}^{3/2} \cap \tilde{\Omega}) \le \frac{3}{2}R = \frac{1}{2} r(\mathcal{C}(\sigma)),
\end{equation}
and so we can similarly conclude that
\begin{equation}\label{f_b_16}
2\beta_B -\beta_B^2 - \frac{\beta_B^2 r(\mathcal{B}(t))}{2(r(\mathcal{C}(\sigma))-r(\mathcal{C}_0))} \ge 0.
\end{equation}
A simple change of variables $t \mapsto t + \sigma$ allows us to rewrite
\begin{equation}\label{f_b_17}
\begin{split}
& \int_0^s \sum_{\substack{\bar{B} \in \mathcal{B} \\ \bar{B} \subset \Omega_{\varepsilon}}} \sum_{B \in \bar{B} \cap \mathcal{B}(t)} \pi d_B^2 \left(2\beta_B -\beta_B^2 - \frac{\beta_B^2 r(\mathcal{B}(t))}{2(r-r(\mathcal{B}_0))} \right)dt  \\
& + \int_0^{\sigma} \sum_{\substack{\bar{B} \in \mathcal{C}(\sigma) \\ \bar{B} \subset \Omega_{\varepsilon}}}\sum_{B \in \bar{B} \cap \mathcal{C}(t)} \pi d_B^2\left(2\beta_B -\beta_B^2 - \frac{\beta_B^2 r(\mathcal{C}(t))}{2 (r(\mathcal{C}(\sigma)) - r(\mathcal{C}_0))} \right)dt \\
&= \sum_{\substack{\bar{B} \in \mathcal{B} \\ \bar{B} \subset \Omega_{\varepsilon}}} \int_0^{s+\sigma}
\sum_{B \in \bar{B} \cap \mathcal{D}(t)} \pi d_B^2 \left( 2\beta_B -\beta_B^2 - \frac{\beta_B^2 r(\mathcal{D}(t))}{2 \lambda(t) } \right)dt,
\end{split}
\end{equation}
where
\begin{equation*}
\lambda(t) = \begin{cases}
        r(\mathcal{C}(\sigma)) - r(\mathcal{C}_0), & t\in[0,\sigma) \\
        r-r(\mathcal{B}_0), & t\in[\sigma,s+\sigma].
             \end{cases}
\end{equation*}

Fix $\bar{B} \in \mathcal{B}$ such that $\bar{B} \subset \Omega_{\varepsilon}$.  For $t \in [0,t_{\bar{B}})$ we have that $\beta_B = 1$, and hence
\begin{equation}\label{f_b_23}
\sum_{B \in \bar{B}\cap \mathcal{D}(t)} \pi d_B^2 \left( 2\beta_B -\beta_B^2 - \frac{\beta_B^2 r(\mathcal{D}(t))}{2 \lambda(t) } \right) \ge \pi d_{\bar{B}} \left(1-\frac{r(\mathcal{D}(t))}{2 \lambda(t)} \right).
\end{equation}
For $t \in [t_{\bar{B}},s+\sigma]$ we similarly estimate
\begin{equation}\label{f_b_19}
\begin{split}
\sum_{B \in \bar{B} \cap \mathcal{D}(t)}d_B^2 \left(2\beta_B - \beta_B^2 -
 \frac{\beta_B^2 r(\mathcal{D}(t))}{2 \lambda(t)} \right)
&=2 \abs{d_{\bar{B}}}^{\frac{1}{2}} \left( \sum\limits_{B \in \bar{B} \cap \mathcal{D}(t)}d_{B}^2       \right)^{\frac{1}{2}} - \abs{d_{\bar{B}}}\left( 1+ \frac{r(\mathcal{D}(t))}{2 \lambda(t)}  \right) \\
&\ge 2 \abs{d_{\bar{B}}}^{\frac{1}{2}} \abs{d_{\bar{B}}}^{\frac{1}{2}} - \abs{d_{\bar{B}}}\left( 1+ \frac{r(\mathcal{D}(t))}{2 \lambda(t)}  \right)\\
&=\abs{d_{\bar{B}}}\left( 1- \frac{r(\mathcal{D}(t))}{2 \lambda(t)}  \right).
\end{split}
\end{equation}
This proves that
\begin{equation}\label{f_b_20}
\begin{split}
  &\sum_{\substack{\bar{B} \in \mathcal{B} \\ \bar{B} \subset \Omega_{\varepsilon}}} \int_0^{s+\sigma}
\sum_{B \in \bar{B} \cap \mathcal{D}(t)} \pi d_B^2 \left( 2\beta_B -\beta_B^2 - \frac{\beta_B^2 r(\mathcal{D}(t))}{2 \lambda(t) } \right)dt \\
& \ge    \pi\sum_{\substack{\bar{B} \in \mathcal{B} \\ \bar{B} \subset \Omega_{\varepsilon}}}
 \abs{d_{\bar{B}}} \int_0^{s+\sigma}\left( 1- \frac{r(\mathcal{D}(t))}{2 \lambda(t)}  \right) dt\\
& = \pi D (s+ \sigma -1),
\end{split}
\end{equation}
where the last equality follows since $r(\mathcal{D}(t))' = r(\mathcal{D}(t))$ for $t \in [0,s+\sigma]\backslash\{\sigma\}$ and $\lambda(t)$ is piecewise constant.

An application of Lemma \ref{radius_bounds} and the bound \eqref{f_b_22} show that
\begin{equation}\label{f_b_11}
r(\mathcal{C}_0) \le \frac{3}{2}r_{\Omega_{\varepsilon}}(\omega_{1/2}^{3/2}).
\end{equation}
Recall that $r(\mathcal{C}(\sigma))= 3r(\mathcal{B}_0)/8$.  This and \eqref{f_b_11} provide the bound
\begin{equation}\label{f_b_21}
\begin{split}
 s+\sigma -1 &= \left(\log{\frac{r}{r(\mathcal{B}_0)}} + \log{\frac{r(\mathcal{C}(\sigma))}{r(\mathcal{C}_0)}}   -1 \right) \\
 & \ge \left( \log{\frac{r}{r(\mathcal{B}_0)}} + \log{\frac{r(\mathcal{B}_0)}{4 r_{\Omega_{\varepsilon}}(\omega_{1/2}^{3/2})}} -1 \right) \\
& = \left( \log{\frac{r}{r_{\Omega_{\varepsilon}}(\omega_{1/2}^{3/2})}}  - C \right).
\end{split}
\end{equation}
Plugging \eqref{f_b_20} and \eqref{f_b_21} into \eqref{f_b_2} yields \eqref{f_b_12}.

\end{proof}

\section{Proof of the main results}\label{thms}

With our technical tools sufficiently developed, we may now
assemble them for use in proving the main theorems.

 We begin with a lemma on the use of the co-area formula in conjunction with sub- and super-level sets.

\begin{lem}\label{co_area_subset}
Let $u:\Omega \rightarrow \mathbb{C}$ and $A: \Omega \rightarrow \Rn{2}$ both be $C^1$ and write (at least locally) $u=\rho v$ with $\rho = \abs{u}.$  Fix $t_0>0$ and $V \subset \Omega$ to be compact.  Then
\begin{equation}\label{c_a_s_1}
\begin{split}
\frac{1}{2}\int_{V \cap \{\rho \ge t_0\}} \rho^2 \abs{\nabla_A v}^2 &= \int_{t_0}^{\infty} -t^2\frac{d}{dt}\left(
\frac{1}{2}\int_{V \cap \{\rho \ge t\}} \abs{\nabla_A v}^2\right)dt \\
& = \frac{t_0^2}{2} \int_{V \cap \{ \rho \ge t_0 \}} \abs{\nabla_A v}^2 + \int_{t_0}^{\infty}2t\left( \frac{1}{2}\int_{V \cap \{\rho \ge t\}} \abs{\nabla_A v}^2 \right)dt
\end{split}
\end{equation}
and
\begin{equation}\label{c_a_s_2}
\begin{split}
\frac{1}{2}\int_{V \cap \{\rho \le t_0\}} \rho^2 \abs{\nabla_A v}^2 &= \int_0^{t_0} -t^2\frac{d}{dt}\left(
\frac{1}{2}\int_{V \cap \{\rho \ge t\} \cap \{\rho \le t_0\}} \abs{\nabla_A v}^2\right)dt \\
& = \int_0^{t_0} 2t\left( \frac{1}{2}\int_{V \cap \{\rho \ge t\} \cap \{ \rho \le t_0\}} \abs{\nabla_A v}^2 \right)dt.
\end{split}
\end{equation}
\end{lem}
\begin{proof}
The first equality in \eqref{c_a_s_1} follows from the co-area formula, and the second follows by integrating by parts.  The same argument proves \eqref{c_a_s_2}.
\end{proof}

\subsection{Proof of Theorem \ref{energy_bound}}
Theorem \ref{energy_bound}  is an improvement on Theorem 4.1 of
\cite{ss_book} that incorporates the $G$ term into the lower
bounds on the vortex balls.  The crucial difference between this
result and those in the previous section is that this one bounds
the energy of the function $u:\Omega \rightarrow \mathbb{C}$,
whereas the previous results were for the $\mathbb{S}^1$-valued
map $v=u/\abs{u}:\Omega \rightarrow \mathbb{S}^1\hookrightarrow\mathbb{C}$. The statement
made in the introduction of Theorem \ref{energy_bound} should be
understood with  $G: \Omega \rightarrow \Rn{2}$  the function
defined in item 4 of Proposition \ref{final_balls} with $\beta_B$
values given by \eqref{f_b_15}.

Proposition \ref{final_balls} produces the collection $\mathcal{B}$ and guarantees items 1 and 2.  The rest of the proof is devoted to showing that \eqref{e_b_15} holds.  By Lemma \ref{polar} we have, writing $u=\rho v$,
\begin{equation}\label{e_b_1}
\begin{split}
\frac{1}{2} \int_V \abs{\nabla_A u}^2 &+ \frac{1}{2\varepsilon^2}(1-\abs{u}^2)^2 + r^2(\curl{A})^2 \\
&= \frac{1}{2} \int_V \abs{\nabla \rho}^2 + \frac{1}{2\varepsilon^2}(1-\rho^2)^2 +\rho^2 \abs{\nabla_A v}^2 + r^2 (\curl{A})^2.
\end{split}
\end{equation}
An application of the co-area formula and integration by parts, the same as that used in Lemma \ref{co_area_subset}, shows that
\begin{equation}\label{e_b_2}
\frac{1}{2} \int_V \rho^2 \abs{\nabla_A v}^2 = \int_0^{\infty} 2t \left( \frac{1}{2} \int_{V\backslash \omega_t}\abs{\nabla_A v}^2\right)dt.
\end{equation}
Then
\begin{equation}\label{e_b_3}
\begin{split}
&\frac{1}{2} \int_V \rho^2 \abs{\nabla_A v}^2 + r^2 (\curl{A})^2 \\
&\ge \int_0^{\infty} 2t \left( \frac{1}{2} \int_{V\backslash \omega_t}\abs{\nabla_A v}^2\right)dt + \int_0^{1-\delta} 2t \left( \frac{r^2}{2} \int_V (\curl{A})^2 \right)dt \\
&= \int_0^{\frac{1}{2}} 2t \left( \frac{1}{2} \int_{V\backslash \omega_t} \abs{\nabla_A v}^2 + \frac{r^2}{2}\int_V (\curl{A})^2  \right)dt + \int_{1-\delta}^{\infty} 2t\left( \frac{1}{2} \int_{V\backslash \omega_t} \abs{\nabla_A v}^2 \right)dt \\
&+ \int_{\frac{1}{2}}^{1-\delta} 2t \left( \frac{1}{2} \int_{V\backslash \omega_t} \abs{\nabla_A v}^2 + \frac{r^2}{2}\int_V (\curl{A})^2  \right)dt \\
&:= A_1 + A_2 + A_3.
\end{split}
\end{equation}
We further break up the first term on the right side of \eqref{e_b_3}:
\begin{equation}\label{e_b_4}
\begin{split}
A_1 &= \int_0^{\frac{1}{2}} 2t \left( \frac{1}{2} \int_{V\backslash \omega_t}
\abs{\nabla_A v}^2 + \frac{r^2}{2}\int_V (\curl{A})^2  \right)dt\\  &= \int_0^{\frac{1}{2}} 2t \left( \frac{1}{2} \int_{\omega_{1/2}^{3/2} \backslash \omega_t} \abs{\nabla_A v}^2 \right) dt
+ \int_0^{\frac{1}{2}} 2t \left( \frac{1}{2} \int_{V\backslash \omega_{1/2}^{3/2}} \abs{\nabla_A v}^2 + \frac{r^2}{2}\int_V (\curl{A})^2  \right)dt.
\end{split}
\end{equation}
Then, by writing $\omega_{1/2}^{3/2} \backslash \omega_t = \omega^{3/2} \cup \omega_{1/2} \backslash \omega_t$,  noting that $\omega_{1/2} \subset V$, and applying \eqref{c_a_s_2} with $t_0 = 1/2$, we may conclude that
\begin{equation}
\begin{split}
\int_0^{\frac{1}{2}} 2t \left( \frac{1}{2} \int_{\omega_{1/2}^{3/2} \backslash
 \omega_t} \abs{\nabla_A v}^2 \right) dt
 &= \int_0^{\frac{1}{2}} 2t \left( \frac{1}{2}\int_{\omega^{3/2}}\abs{\nabla_A v}^2 + \frac{1}{2}\int_{\omega_{1/2}\backslash \omega_t}\abs{\nabla_A v}^2 \right)dt\\
 &=  \frac{1}{8} \int_{\omega^{3/2}}\abs{\nabla_A v}^2 + \frac{1}{2} \int_{\omega_{1/2}}\rho^2 \abs{\nabla_A v}^2.
\end{split}
\end{equation}
Since the integrand does not depend on $t$, we have
\begin{equation}
\int_0^{\frac{1}{2}} 2t \left( \frac{1}{2} \int_{V\backslash
\omega_{1/2}^{3/2}} \abs{\nabla_A v}^2 + \frac{r^2}{2}\int_V
(\curl{A})^2  \right)dt  = \frac{1}{4}\left( \frac{1}{2}
\int_{V\backslash \omega_{1/2}^{3/2}} \abs{\nabla_A v}^2 +
\frac{r^2}{2}\int_V (\curl{A})^2 \right).
\end{equation}
From \eqref{c_a_s_1}, applied with $t_0 = 1-\delta$, we bound the second term in \eqref{e_b_3}
\begin{equation}\label{e_b_5}
\begin{split}
A_2& = \int_{1-\delta}^{\infty} 2t\left( \frac{1}{2} \int_{V\backslash \omega_t} \abs{\nabla_A v}^2 \right)dt = \frac{1}{2} \int_{V\backslash \omega_{1-\delta}} (\rho^2 - (1-\delta)^2) \abs{\nabla_A v}^2 \\
& \ge \frac{1}{2} \int_{\omega^{3/2}} (\rho^2 - 1) \abs{\nabla_A v}^2.
\end{split}
\end{equation}
When $\rho \ge \frac{3}{2}$, the inequality $ \rho^2 - \frac{3}{4} \ge \frac{2}{3} \rho^2$ holds; hence,
\begin{equation}\label{e_b_6}
\frac{1}{2} \int_{\omega^{3/2}} (\rho^2 - 1) \abs{\nabla_A v}^2 + \frac{1}{8} \int_{\omega^{3/2}}\abs{\nabla_A v}^2 \ge \frac{1}{3} \int_{\omega^{3/2}} \rho^2 \abs{\nabla_A v}^2.
\end{equation}
We now combine \eqref{e_b_3} -- \eqref{e_b_6}, leaving $A_3$ as it was, and arrive at the bound
\begin{multline}\label{e_b_7}
\frac{1}{2} \int_V \rho^2 \abs{\nabla_A v}^2 + r^2 (\curl{A})^2  \ge
\frac{1}{4}\left( \frac{1}{2} \int_{V\backslash \omega_{1/2}^{3/2}} \abs{\nabla_A v}^2 + \frac{r^2}{2}\int_V (\curl{A})^2 \right) \\
+ \int_{\frac{1}{2}}^{1-\delta} 2t \left( \frac{1}{2} \int_{V\backslash \omega_t} \abs{\nabla_A v}^2 + \frac{r^2}{2}\int_V (\curl{A})^2  \right)dt
+ \frac{1}{3} \int_{\omega_{1/2}^{3/2}} \rho^2 \abs{\nabla_A v}^2.
\end{multline}

Recalling the notation
\begin{equation*}
F_{\varepsilon}(\rho,V) = \frac{1}{2} \int_V \abs{\nabla \rho}^2 + \frac{1}{2\varepsilon^2}(1-\rho^2)^2
\end{equation*}
and the decomposition \eqref{e_b_1}, we can use \eqref{e_b_7} to see that
\begin{equation}\label{e_b_001}
\begin{split}
\frac{1}{2} \int_V \abs{\nabla_A u}^2 + \frac{1}{2\varepsilon^2}(1-\abs{u}^2)^2 + r^2(\curl{A})^2
 &= F_\varepsilon(\rho,V) + \frac{1}{2} \int_V \rho^2 \abs{\nabla_A v}^2 + r^2 (\curl{A})^2 \\
& \ge B_1 + B_2 + B_3,
\end{split}
\end{equation}
where
\begin{equation*}
B_1 := \frac{1}{4}\left( F_\varepsilon(\rho,V) + \frac{1}{2} \int_{V\backslash \omega_{1/2}^{3/2}} \abs{\nabla_A v}^2 + \frac{r^2}{2}\int_V (\curl{A})^2 \right),
\end{equation*}
\begin{equation*}
B_2 := \frac{3\beta}{4} F_\varepsilon(\rho,V) + \int_{\frac{1}{2}}^{1-\delta} 2t \left( \frac{1}{2} \int_{V\backslash \omega_t} \abs{\nabla_A v}^2 + \frac{r^2}{2}\int_V (\curl{A})^2  \right)dt,
\end{equation*}
\begin{equation*}
B_3 :=  \frac{3(1-\beta)}{4}F_\varepsilon(\rho,V) + \frac{1}{3} \int_{\omega_{1/2}^{3/2}} \rho^2 \abs{\nabla_A v}^2,
\end{equation*}
and $\beta \in(0,1)$ is to be chosen later in the proof.

To bound $B_1$, we employ Proposition \ref{final_balls_part_2} to see that
\begin{equation}\label{e_b_8}
\begin{split}
&\frac{1}{2} \int_{V\backslash \omega_{1/2}^{3/2}} \abs{\nabla_A v}^2 + \frac{r^2}{2}\int_V (\curl{A})^2 + \frac{1}{2} \int_V \abs{\nabla \rho}^2 + \frac{1}{2\varepsilon^2}(1-\rho^2)^2 \\
&\ge \frac{1}{2}\int_{V \backslash \omega_{1/2}^{3/2}} \abs{\nabla_{A+G} v}^2 + \pi D \left(\log{\frac{r}{r_{\Omega_{\varepsilon}}(\omega_{1/2}^{3/2})}}  - C \right)  + F_{\varepsilon}(\rho,V). \\
\end{split}
\end{equation}
Then, an application of Lemma \ref{e_rad_bound} shows that
\begin{equation}\label{e_b_8_0}
\begin{split}
\pi D
\left(\log{\frac{r}{r_{\Omega_{\varepsilon}}(\omega_{1/2}^{3/2})}}
- C \right) + F_{\varepsilon}(\rho,V)
&\ge \pi D \left(\log{\frac{r}{C\varepsilon F_{\varepsilon}(\rho,V)}}  - C \right)  + F_{\varepsilon}(\rho,V) \\
&\ge \pi D \left(\log{\frac{r}{\varepsilon D}} - C \right) + F_{\varepsilon}(\rho,V)- \pi D \log{\frac{F_{\varepsilon}(\rho,V)}{\pi D}} \\
&\ge \pi D \left(\log{\frac{r}{\varepsilon D}} - C \right),
\end{split}
\end{equation}
where the last line follows from the inequality $x-a\log{\frac{x}{a}} \ge 0$.  On the set $V \backslash \omega_{1/2}^{3/2}$ it is the case that $1/2 \le \rho \le 3/2$, and so $1 \ge 4 \rho^2/9$.  Hence, from this bound, \eqref{e_b_8}, and \eqref{e_b_8_0}, we may conclude that
\begin{equation}\label{e_b_002}
\begin{split}
B_1 &\ge \frac{1}{4} \left( \frac{1}{2}\int_{V \backslash \omega_{1/2}^{3/2}} \abs{\nabla_{A+G} v}^2 + \pi D \left(\log{\frac{r}{\varepsilon D}} - C \right)      \right) \\
& \ge \frac{1}{18}\int_{V \backslash \omega_{1/2}^{3/2}} \rho^2 \abs{\nabla_{A+G} v}^2 + \frac{\pi D}{4}\left(\log{\frac{r}{\varepsilon D}} - C \right).
\end{split}
\end{equation}

To control $B_2$, we begin by using \eqref{erb_1} and Lemma \ref{haus_bound} to find the bound
\begin{equation}\label{e_b_9}
\begin{split}
\frac{3\beta}{4} F_{\varepsilon}(\rho,V) &\ge \frac{3\sqrt{2}\beta}{8\varepsilon} \int_0^{\infty} \abs{1-t^2} \mathcal{H}^1(\{\rho=t \})dt \\
&\ge \frac{3\sqrt{2}\beta}{4\varepsilon} \int_{\frac{1}{2}}^{1-\delta}(1-t^2)r_{\Omega_{\varepsilon}}(\omega_t)dt.
\end{split}
\end{equation}
Then \eqref{e_b_9} and \eqref{f_b_1} prove that
\begin{equation}\label{e_b_10}
B_2  \ge \int_{\frac{1}{2}}^{1-\delta} \left(2t \pi D \left( \log{\frac{r}{r_{\Omega_{\varepsilon}}(\omega_t)}} -C\right)
+ \frac{3\sqrt{2} \beta}{4\varepsilon} (1-t^2) r_{\Omega_{\varepsilon}}(\omega_t)\right) \;dt.
\end{equation}
As $r_{\Omega_{\varepsilon}}(\omega_t)$ varies, the integrand on the right hand side of \eqref{e_b_10} achieves its minimum at
\begin{equation*}
r_{\Omega_{\varepsilon}}(\omega_t) = \frac{8\pi D t \varepsilon}{3\sqrt{2}\beta(1-t^2)}.
\end{equation*}
Plugging this in, we get the estimate
\begin{equation}\label{e_b_11}
\begin{split}
B_2 &\ge \int_{\frac{1}{2}}^{1-\delta} 2\pi D t\left( \log{\frac{3\sqrt{2} r \beta (1-t^2)}{8 \pi D t \varepsilon}} -C +1 \right)dt \\
&= \int_{\frac{1}{2}}^{1-\delta} 2\pi D t \left( \log{\frac{r}{\varepsilon D}}  + \log{\frac{3\sqrt{2} \beta (1-t^2)}{8 \pi t}} - C \right)dt \\
&= \pi D \left( \left((1-\delta)^2-\frac{1}{4} \right)\log{\frac{r}{\varepsilon D}} -C  \right).
\end{split}
\end{equation}

We now choose $\beta=\frac{23}{27}$ so that $\frac{3(1-\beta)}{8}=\frac{1}{18}$.  Then
\begin{equation}\label{e_b_003}
B_1 + B_3 \ge \frac{1}{18}\int_{V } \abs{\nabla_{A+G} u}^2 + \frac{1}{2\varepsilon^2}(1-\abs{u}^2)^2 +  \frac{\pi D}{4}\left(\log{\frac{r}{\varepsilon D}} - C \right).
\end{equation}
Using \eqref{e_b_11} and \eqref{e_b_003} in \eqref{e_b_001} then shows that
\begin{equation}\label{e_b_12}
\begin{split}
&\frac{1}{2} \int_V \abs{\nabla_A u}^2 + \frac{1}{2\varepsilon^2}(1-\abs{u}^2)^2 + r^2(\curl{A})^2 \\
&\ge  \pi D \left( (1-\delta)^2\log{\frac{r}{\varepsilon D}} -C  \right) + \frac{1}{18}\int_{V } \abs{\nabla_{A+G} u}^2 + \frac{1}{2\varepsilon^2}(1-\abs{u}^2)^2 .
\end{split}
\end{equation}
Now, by assumption $r\le 1 \le D$, so $\log{\frac{r}{D}}\le 0$.  Since $\delta = \varepsilon^{\alpha/4}$, we have that for $\varepsilon \le \varepsilon_0=\varepsilon_0(\alpha)$, the inequalities
\begin{equation}\label{e_b_13}
\begin{split}
& \delta^2 - \delta \le 0 \\
& (2\delta -\delta^2)\log{\varepsilon} \ge -1
\end{split}
\end{equation}
both hold.  Hence, for $\varepsilon \le \varepsilon_0$,
\begin{equation}\label{e_b_14}
\begin{split}
(1-\delta)^2\log{\frac{r}{\varepsilon D}} - C &= \log{\frac{r}{\varepsilon D}}-C + (\delta^2-2\delta)\log{\frac{r}{D}} +(2\delta -\delta^2)\log{\varepsilon} \\
& \ge \log{\frac{r}{\varepsilon D}}-C -1.
\end{split}
\end{equation}
Combining \eqref{e_b_12} with \eqref{e_b_14} gives \eqref{e_b_15}.

\qed

\subsection{Proof of Theorem \ref{norm_switch} and corollaries}
\begin{proof}[Proof of Theorem \ref{norm_switch}]
Theorem \ref{norm_switch}  justifies the  selection of the
function $G$.  It has been chosen so  that $\|G\|_{\lti}$ only
depends on the final data of Theorem \ref{energy_bound}, that is
on natural quantities.
  This
estimate of $\|G\|_{\lti}$, Proposition \ref{g_norm}, is quite
technical and is thus reserved for the next section.  A more thorough
discussion of the space $L^{2,\infty}$, also known as weak-$L^2$,
is also reserved for the next section.

We begin by noting that $\nabla_A u = \nabla_{A+G} u + iGu$.  This and the fact that $\pqnormspace{f}{2}{\infty}{V} \le \pqnormspace{g}{2}{\infty}{V}$ if $\abs{f} \le \abs{g}$ allow us to estimate
\begin{equation}\label{n_s_2}
\begin{split}
\frac{1}{2} \pqnormspace{\nabla_A u}{2}{\infty}{V}^2 &\le \pqnormspace{\nabla_{A+G} u}{2}{\infty}{V}^2 + \pqnormspace{iGu}{2}{\infty}{V}^2 \\
& \le \pnormspace{\nabla_{A+G} u}{2}{V}^2 + \frac{9}{4}\pqnormspace{G}{2}{\infty}{V}^2.
\end{split}
\end{equation}
The second inequality follows since $\abs{u} \le \frac{3}{2}$ on the support of $G$.  Write
\begin{equation*}
 F_{\varepsilon}^r(u,A,V) = \frac{1}{2} \int_V \abs{\nabla_A u}^2 + \frac{1}{2\varepsilon^2}(1-\abs{u}^2)^2 + r^2(\curl{A})^2.
\end{equation*}
We now employ Theorem \ref{energy_bound} to bound
\begin{equation}\label{n_s_3}
\begin{split}
\pnormspace{\nabla_{A+G} u}{2}{V}^2  \le 18\left( F_{\varepsilon}^r(u,A,V) - \pi D \left( \log{\frac{r}{\varepsilon D}} -C  \right) \right).
\end{split}
\end{equation}
We will show in Proposition \ref{g_norm} that
\begin{equation}\label{n_s_4}
\begin{split}
\frac{9}{4}\pqnormspace{G}{2}{\infty}{V}^2 &\le  \frac{216(1+\eta)}{2\eta -1}\left( F_{\varepsilon}^r(u,A,V) - \pi D \left(\log{\frac{r}{\varepsilon D}}-C \right)  \right) \\
&+ \pi \frac{9(1+\eta)}{1-\eta} \sum_{\substack{\bar{B} \in \mathcal{B} \\ \bar{B} \subset \Omega_{\varepsilon}}} d_{\bar{B}}^2.
\end{split}
\end{equation}
Now choose $\eta = \frac{5+\sqrt{2785}}{60}\approx .962$ so that
\begin{equation*}
 18 + \frac{216(1+\eta)}{2\eta -1} = \frac{9(1+\eta)}{1-\eta}.
\end{equation*}
Combining \eqref{n_s_2} -- \eqref{n_s_4} yields \eqref{n_s_1} with constant $C = (1-\eta)/(18(1+\eta)) \approx 1/951$.

\end{proof}

The previous theorem dealt with the energy content of the set $V
\subset \Omega$. We can deduce a slightly stronger version of
Corollary \ref{coro1}.

\begin{cor}\label{norm_switch_domain}
 Assume the hypotheses of Theorem \ref{energy_bound}.  Then
 \begin{equation}
  C\pqnormspace{\nabla_A u}{2}{\infty}{\Omega}^2 \le
   F_\varepsilon^r(u,A,\Omega) - \pi D  \left( \log{\frac{r}{\varepsilon D}}
   -C  \right) +  \pi \sum_{\substack{B \in \mathcal{B} \\
   B \subset \Omega_{\varepsilon}}} d_{B}^2.
 \end{equation}
\end{cor}
\begin{proof}
 Add $F_\varepsilon^r(u,A,\Omega \backslash V)$ to both sides of \eqref{n_s_1}.  We then bound
 \begin{equation}
 \begin{split}
  &C\pqnormspace{\nabla_A u}{2}{\infty}{V}^2 + F_\varepsilon^r(u,A,\Omega \backslash V) \\
  & \ge C\pqnormspace{\nabla_A u}{2}{\infty}{V}^2 + \pnormspace{\nabla_A u}{2}{\Omega \backslash V}^2 \\
&\ge C\pqnormspace{\nabla_A u}{2}{\infty}{V}^2 + \pqnormspace{\nabla_A u}{2}{\infty}{\Omega \backslash V}^2 \\
& \ge C\pqnormspace{\nabla_A u}{2}{\infty}{\Omega}^2,
 \end{split}
 \end{equation}
where the last inequality follows by using the convexity of norms, and $C$ is a different constant.  The result follows.
\end{proof}

\begin{proof}[Proof of Proposition \ref{corobbh}]
It is proved in Theorem 0.5 of \cite{bbh} that minimizers of $E_\ep$
with this constraint have  exactly $d$ zeroes  of degree $1$
 which converge to $d$ distinct
points $a_1, \dotsc, a_d$, minimizing $W_g$.
They also prove that their energy is
\begin{equation}\label{nrjbbh}
\min E_\ep= \pi d \lep + \min W_g + d \gamma +o(1),\end{equation}
where $\gamma$ is a universal constant. Let us  apply the
vortex-ball construction to these solutions, choosing for final
radius $r = \frac{1}{4} \min_{i,j} \left( \dist(a_i,\partial \Omega), \abs{a_i - a_j} \right).$
Since the final balls $B \in \mathcal{B}$ cover all
the zeroes of $u$, and there is exactly one zero $b_i^\ep$ with
nonzero degree, converging to each $a_i$, there is  one
ball $B_i$ in the collection containing $b_i^\ep$. Since
$ d_i= \deg (u_\ep, \p B_i)=1, $
 and there are no other zeroes of $u_\ep$, we have  $D=
d$ (with our previous notation) and  Corollary \ref{coro1} (taken
with $A\equiv 0$) gives us
$$E_\ep(u_\ep) + \pi d \ge C \pqnormspace{\nabla u_\varepsilon}{2}{\infty}{\Omega}^2 + \pi d (\lep -
C - \log d),$$ where $C$ is a universal constant. In view of
\eqref{nrjbbh}, this implies that
$$ C \pqnormspace{\nabla u_\varepsilon}{2}{\infty}{\Omega}^2 \le \min W_g +  d \gamma + C d  + \pi d
\log d + o(1), $$ and the first  result follows.

Since $\lti $ is a dual Banach space, we deduce from this bound
that, as $\ep \to 0$, up to extraction, $\nab u_\ep$ converges
weakly-$*$ in $\lti $, to its distributional limit. But it is proved
in \cite{bbh} that $\nab u_\ep \to \nab u_\star $ uniformly away
from $a_1, \cdots, a_d$ (in fact in $C^k_{loc}$), where $u_\star $
is given by $$ u_\star(x)= e^{i H(x)}   \prod_{k=1}^d
\frac{x-a_k}{|x-a_k|}$$ with $H$ a harmonic function. Note in
particular that $u_\star\in W^{1,p}(\om)$ for $p<2$.

We claim that $\nab u_\ep \to \nab u_\star $ in the sense of
distributions on $\om$. Indeed, let $X$ be a smooth compactly
supported test vector field. Fix $\ro>0$ and let us write
$$\io (\nab u_\ep  - \nab u_\star) \cdot X = \int_{\om\backslash
\cup_i B(a_i, \ro)} (\nab u_\ep  - \nab u_\star) \cdot X + \sum_i
\int_{B(a_i,\ro)} (\nab u_\ep  - \nab u_\star) \cdot X. $$ The
first term in the right-hand side tends to $0$ by uniform
convergence of $\nab u_\ep $ to $\nab u_\star$ away from the
$a_i$'s. The second term is bounded by H\"older's inequality by
$C\|X\|_{L^\infty} \|\nab u_\ep- \nab u_\star \|_{L^p(\om)}
\ro^{2/q}$, where $p<2 $ and $1/p+ 1/q=1$. This is bounded by $C
\ro^{2/q}\|X\|_{L^\infty}$ since $\nab u_\star \in L^p(\om)$ for
all $p<2$ and $\nab u_\ep $ is bounded in $L^p(\om)$ for all $p<2$
($\lti(\om)$ embeds in $L^p(\om)$ for all $p<2$). Letting  $\ro $
tend to $0$ we conclude that  $\io (\nab u_\ep- \nab u_\star)\cdot
X\to 0$ and finally that $\nab u_\ep \rightharpoonup \nab u_\star
$ weakly-$*$ in $\lti(\om)$.

\end{proof}

\section{The $L^{2,\infty}$ norm of $G$}\label{l2_inf_of_G}

\subsection{Definitions and preliminary results}\label{l_2_inf_discussion}
We begin with a discussion of the various quantities needed to define and norm the space $L^{2,\infty}$.  For a function $f:\Omega \rightarrow \Rn{k}$, $k\ge 1$, we define the distribution function of $f$ by
\begin{equation}
\lambda_f(t) = \abs{\{x \in \Omega \;|\; \abs{f(x)} > t   \}}.
\end{equation}
This allows us to define the decreasing rearrangement of $f$ as $f^{*}:\Rn{+} \rightarrow \Rn{+}$, where
\begin{equation}\label{dec_re_def}
f^{*}(t) = \inf \{ s>0  \;|\;  \lambda_f(s) \le t \}.
\end{equation}
We then define the quantity
\begin{equation}
\wnorm{f} = \sqrt{\sup_{t>0}t^2\lambda_f(t)} = \sup_{t>0} t
\lambda_f(t)^{\frac{1}{2}} = \sup_{t>0} t^{\frac{1}{2}} f^{*}(t),
\end{equation}
and $L^{2,\infty}(\Omega) = \{f \;|\; \wnorm{f} < \infty \}$.  Unfortunately, this does not define a norm, but rather a quasi-norm.  That is, $\wnorm{\cdot}$ satisfies
\begin{equation*}
 \begin{cases}
  \wnorm{\alpha f} = \abs{\alpha} \wnorm{f} \\
  \wnorm{f} = 0 \text{ if and only if } f=0 \text{ a.e.} \\
  \wnorm{f+g} \le C(\wnorm{f} + \wnorm{g}) \text{ for some } C\ge1.
 \end{cases}
\end{equation*}
It can be shown that with $\wnorm{\cdot}$, $L^{2,\infty}$ is a quasi-Banach space, i.e. a linear space in which every quasi-norm Cauchy sequence converges in the quasi-norm.  However, as the next lemma shows, the space can, in fact, be normed.  We define
\begin{equation}
\begin{split}
  \pqnorm{f}{2}{\infty} & = \sup_{\abs{E}<\infty} \abs{E}^{-1/2} \int_{E} \abs{f(x)}dx \\
                        & = \sup_{t>0}\frac{1}{t^{\frac{1}{2}}} \sup_{\abs{E}=t} \int_E \abs{f(x)}dx\\
                        & = \sup_{t>0}\frac{1}{t^{\frac{1}{2}}} \int_0^t f^{*}(s)ds,
\end{split}
\end{equation}
which is obviously a norm.

\begin{lem}\label{l2_inf_normability}
$L^{2,\infty}$ is a Banach space with norm $\pqnorm{\cdot}{2}{\infty}$, and
\begin{equation}
\wnorm{f} \le \pqnorm{f}{2}{\infty} \le 2 \wnorm{f}.
\end{equation}
\end{lem}
\begin{proof}
Since $f^{*}$ is decreasing, we see that
\begin{equation}
\pqnorm{f}{2}{\infty} = \sup_{t>0}\frac{1}{t^{\frac{1}{2}}} \int_0^t f^{*}(s)ds \ge \sup_{t>0} \frac{1}{t^{\frac{1}{2}}} t f^{*}(t) = \sup_{t>0} t^{\frac{1}{2}} f^{*}(t) = \wnorm{f}.
\end{equation}
For the second inequality we note that
\begin{equation}
\frac{1}{t^{\frac{1}{2}}} \int_0^t f^{*}(s)ds = \frac{1}{t^{\frac{1}{2}}} \int_0^t (s^{\frac{1}{2}}f^{*}(s))\frac{ds}{s^{\frac{1}{2}}} \le \wnorm{f} \frac{2t^{\frac{1}{2}}}{t^{\frac{1}{2}}}=2\wnorm{f}.
\end{equation}
This also shows how to construct a function that makes the inequalities sharp: any $f$ so that $f^{*}(s) = \frac{c}{\sqrt{s}}$ will do.  This is the case for $f(x) = 1/\abs{x}$ in $\Rn{2}$.
\end{proof}

We now present the

\begin{proof}[Proof of Proposition \ref{proplb}]
First rewrite the $L^2$ integral using the distribution function:
\begin{equation}
 \int_{\Omega} \abs{f}^2 = \int_0^{\infty} 2t \lambda_f(t) dt.
\end{equation}
We break this integral into two parts and utilize the boundedness
of $f$ and the trivial inequality $\lambda_f(t) \le \abs{\Omega}$
for all $t>0$.  Indeed,
\begin{equation}
\begin{split}
 \int_0^{\infty} 2t \lambda_f(t) dt &= \int_0^C 2t \lambda_f(t) dt + \int_C^{\frac{C}{\varepsilon}} 2t\lambda_f(t) dt\\
 & \le \abs{\Omega} \int_0^C 2t dt + 2 \sup_{t>0}(t^2 \lambda_f(t)) \int_C^{\frac{C}{\varepsilon}} \frac{dt}{t} \\
 & \le \abs{\Omega} C^2  + 2 \pqnormspace{f}{2}{\infty}{\Omega}^2 \log{\frac{C}{C \varepsilon}},
\end{split}
\end{equation}
where we have used Lemma \ref{l2_inf_normability} in the last inequality.  The result follows by dividing both sides by $2
\abs{\log{\varepsilon}}$.
\end{proof}
\subsection{The calculation}
Before proving the main result we prove some quasi-norm estimates for simplified versions of $G$.  The main result breaks $G$ into various simplified components in order to utilize these estimates.

\begin{lem}\label{annuli_exp_estimate}
Suppose we are given a collection of disjoint annuli $\{A_i\}$, $i=1,\dotsc,n$, where
\begin{equation*}
A_i = \{r_i < \abs{x-c_i}\le s_i \} \subset \Rn{2},
\end{equation*}
$c_i$ denotes the center of $A_i$, and $r_i$ and $s_i$ are the inner and outer radii respectively.  Let
\begin{equation}
f(x) = \sum_{i=1}^n \text{\large{$\chi$}}_{A_i}(x) v_i(x) \frac{a_i}{\abs{x-c_i}},
\end{equation}
where $v_i: A_i \rightarrow \Rn{k}$ is a vector field so that $\abs{v_i}=1$ and $a_i$ is a constant for $i=1,\dotsc,n$.
Write $\tau_i = \log{\frac{s_i}{r_i}}$ for the conformal factor of $A_i$.  Then for $t>0$,
\begin{equation}
t^2 \lambda_f(t) \le \pi \sum_{i=1}^n a_i^2 \left( 1-e^{-2\tau_i}
\right).
\end{equation}
\end{lem}

\begin{proof}
We begin by noting that on the annulus $A_i$ it is the case that
\begin{equation}
\frac{\abs{a_i}}{s_i} \le \abs{f} < \frac{\abs{a_i}}{r_i}.
\end{equation}
Then for any $t>0$ and any annulus $A_i$, the measure of the set in $A_i$ where $f>t$ is simple to calculate.  Indeed, if $t\le \abs{a_i}/s_i$, then $f>t$ on the whole annulus, which has measure $\pi(s_i^2 - r_i^2)$.  If $t\ge \abs{a_i}/r_i$, then $f<t$ everywhere on the annulus, and so the measure is zero.  Finally, if $\abs{a_i}/s_i < t < \abs{a_i}/r_i$, then $f>t$ exactly on the subannulus $\{ r_i < \abs{x-c_i} \le \rho_i  \}$, where
\begin{equation}
 \rho_i = \frac{\abs{a_i}}{t},
\end{equation}
which has measure $\pi (a_i^2/t^2 - r_i^2 )$.

Combining these, for any $t>0$ we may then write
\begin{equation}
\lambda_f(t) = \sum_{\{i\;\vert\;\frac{\abs{a_i}}{s_i} <t <
\frac{\abs{a_i}}{r_i} \}} \pi \left(\frac{a_i^2}{t^2}-r_i^2
\right)    + \sum_{\{i\;|\; t \le \frac{\abs{a_i}}{s_i}\}} \pi
(s_i^2-r_i^2).
\end{equation}
Then
\begin{equation}
\begin{split}
  t^2\lambda_f(t) & =  \sum_{\{i\;\vert\;\frac{\abs{a_i}}{s_i} <t < \frac{\abs{a_i}}{r_i} \}} \pi (a_i^2-t^2r_i^2) +   \sum_{\{i\;|\; t \le \frac{\abs{a_i}}{s_i}\}} \pi (s_i^2-r_i^2)t^2 \\
&\le \sum_{\{i\;|\;\frac{\abs{a_i}}{s_i} <t < \frac{\abs{a_i}}{r_i} \}} \pi a_i^2 \left(1 - \frac{r_i^2}{s_i^2} \right) + \sum_{\{i\;|\; t \le \frac{\abs{a_i}}{s_i}\}} \pi a_i^2 \left(1-\frac{r_i^2}{s_i^2}\right) \\
&\le \sum_{i=1}^n \pi a_i^2 \left(1-\frac{r_i^2}{s_i^2}\right).
\end{split}
\end{equation}
Plugging in $\tau_i = \log{\frac{s_i}{r_i}}$ proves the result.

\end{proof}

The next lemma tells us that a collection of annuli with uniformly bounded degrees and the property that they can be rearranged to fit concentrically inside each other can, for the purposes of estimating the $L^{2,\infty}$ quasi-norm, be regarded as a single annulus.

\begin{lem}\label{zero_merging_grouping}
Suppose $\{A_i\}$, $i=1,\dotsc,n$, is a collection of disjoint annuli, where
\begin{equation*}
A_i = \{r_i < \abs{x-c_i}\le s_i \} \subset \Rn{2},
\end{equation*}
$c_i$ denotes the center of $A_i$, and $r_i$ and $s_i$ are the inner and outer radii respectively.  Suppose further that the annuli can be arranged concentrically without overlap.  That is, suppose that
\begin{equation*}
r_1 < s_1 \le r_2 < s_2 \le r_3 \le \dotsb < s_{n-1} \le r_n < s_n.
\end{equation*}
Let
\begin{equation}
f(x) = \sum_{i=1}^n \text{\large{$\chi$}}_{A_i}(x) v_i(x) \frac{a_i}{\abs{x-c_i}},
\end{equation}
where the $a_i$ are constants such that $\abs{a_i}\le \abs{a}$ and $v_i: A_i \rightarrow \Rn{k}$ is a vector field so that $\abs{v_i}=1$ for $i=1,\dotsc,n$.
Then
\begin{equation}
t^2 \lambda_f(t) = t^2 \sum_{i=1}^n \abs{A_i \cap \{\abs{f}>t \} }
\le \pi a^2.
\end{equation}
\end{lem}

\begin{proof}
Since the distribution function is invariant under translations, without loss of generality we may assume that the annuli are concentric with common center $c$.  This reduces $f$ to the form
\begin{equation}
 f(x) = \sum_{i=1}^n \text{\large{$\chi$}}_{A_i}(x) v_i(x) \frac{a_i}{\abs{x-c}}.
\end{equation}

Consider the function
\begin{equation}
g(x) = \frac{a e_1}{\abs{x-c}},
\end{equation}
where $e_1 = (1,0,\dotsc,0)\in \Rn{k}$. The pointwise bound
$\abs{f(x)} \le \abs{g(x)}$ yields the bound $\lambda_{f}(t) \le
 \lambda_g(t)$ for all $t>0$.  It is a simple matter to see that
\begin{equation}
 \lambda_g(t) = \pi \frac{a^2}{t^2},
\end{equation}
and hence,
\begin{equation}
t^2 \lambda_f(t) \le t^2 \lambda_g(t) = \pi a^2.
\end{equation}
\end{proof}

We are now ready to prove the main result of this section.

\begin{prop}\label{g_norm}
Let $G: \Omega \rightarrow \Rn{2}$ be the function defined in Proposition \ref{final_balls} with $\eta\in(0,1)$ fixed and the $\beta_B$ values given by \eqref{f_b_15}.   Write
\begin{equation*}
 F_{\varepsilon}^r(u,A,V) = \frac{1}{2} \int_V \abs{\nabla_A u}^2 + \frac{1}{2\varepsilon^2}(1-\abs{u}^2)^2 + r^2(\curl{A})^2.
\end{equation*}
Then
\begin{equation}\label{g_n_30}
\begin{split}
\pqnormspace{G}{2}{\infty}{V}^2 &\le  \frac{96(1+\eta)}{2\eta -1}\left( F_{\varepsilon}^r(u,A,V) - \pi D \left(\log{\frac{r}{\varepsilon D}}-C \right)  \right) \\
&+ \pi \frac{4(1+\eta)}{1-\eta} \sum_{\substack{\bar{B} \in \mathcal{B} \\ \bar{B} \subset \Omega_{\varepsilon}}} d_{\bar{B}}^2.
\end{split}
\end{equation}
\end{prop}
\begin{proof}
Step 1

To begin we must translate the notation used to define $G$ into different notation that is more cumbersome but that will allow a more exact enumeration of the objects generated by the ball construction.  Recall that to define $G$, the collection $\{\mathcal{D}(t)\}_{t\in[0,s+\sigma]}$ defined by \eqref{f_b_14} is refined to the subcollection $\{\mathcal{G}(t)\}_{t\in[0,s+\sigma]}$ that consists of all balls that stay entirely inside $\Omega_{\varepsilon}$.  Let $N$ be the number of balls in $\mathcal{G}(s+\sigma)= \{\bar{B}_1,\dotsc,\bar{B}_N\}$, i.e. the number of final balls.  Let $T$ be the finite set of merging times in the growth of $\mathcal{G}(t)$, where here we count $t=\sigma$, the time when the collection shifts from $\mathcal{C}(\sigma)$ to $\mathcal{B}(0)$, as a merging time.  Let $0=t_0 <t_1 < \dotsb < t_{K-1} < t_K= s+\sigma$ be an enumeration of $T \cup \{0,s+\sigma\}$.  For $k=1,\dotsc,K$ and $t\in[t_{k-1},t_k)$ we call all balls in $\mathcal{G}(t)$ members of the $k^{th}$ generation.  We write $\mathcal{G}(t_k^-)$ for the collection of balls obtained as $t\rightarrow t_k^-$, i.e. the collection of pre-merged balls at time $t=t_k$. Similarly, when we write $\mathcal{G}(t_k)$ we refer to the post-merged balls.  For $k=1,\dotsc,K$ and $n=1,\dotsc,N$ we enumerate
\begin{equation*}
\begin{split}
&\{B_{i,k,n} \}_{i=1}^{M_{k,n}} = \{ B \in \mathcal{G}(t_k^-) \;|\; B \subset \bar{B}_n \}, \text{ and } \\
&\{\tilde{B}_{i,k,n} \}_{i=1}^{M_{k,n}} = \{ B \in \mathcal{G}(t_{k-1}) \;|\; B \subset \bar{B}_n \},
\end{split}
\end{equation*}
in such a way that $\tilde{B}_{i,k,n} \subset B_{i,k,n}$.  We define the annuli $A_{i,k,n} = B_{i,k,n} \backslash \tilde{B}_{i,k,n}$ and write $d_{i,k,n} = \deg(u,\partial B_{i,k,n})$ for the degree of $u$ in the annulus $A_{i,k,n}$.  For fixed $k=1,\dotsc,K$ we say the annuli $\{A_{i,k,n} \}$ are $k^{th}$ generation annuli.  Without loss of generality we may assume that the indices are ordered so that $\abs{d_{i,k,n}}$ is a decreasing sequence with respect to $i$ for $k$ and $n$ fixed.  Write $D_n = d_{\bar{B}_n}$.  We define the conformal growth factor in the $k^{th}$ generation, denoted $\tau_k$, by
\begin{equation*}
\tau_k = \log{\frac{r(\mathcal{G}(t_k^-))}{r(\mathcal{G}(t_{k-1}))}}.
\end{equation*}

Recall that for each $\bar{B}_n$, $n=1,\dotsc,N$, Corollary \ref{degree_analysis_firstgen} provides a transition time $t_{\bar{B}_n}$ (depending on $\eta$).  In the current setting, the more natural notion is that of transition generation, and in fact, the proof of Lemma \ref{degree_analysis} shows that the transition time actually occurs at one of the $t_k$ for $k=0,\dotsc,K-1$.  We then define the transition generation $k_n$ as the unique $k$ such that $t_{\bar{B}_n} \in [t_{k-1}, t_k)$.  If we define generational versions of the negative and positive vorticity masses $N(t)$ and $P(t)$ from \eqref{d_a_2} by
\begin{equation*}
\begin{split}
& N(k,n):= \sum_{\substack{1 \le i \le M_{k,n} \\ d_{i,k,n}< 0}} \abs{d_{i,k,n}}  \\
& P(k,n):= \sum_{\substack{1 \le i \le M_{k,n} \\ d_{i,k,n}\ge0}} d_{i,k,n},
\end{split}
\end{equation*}
then the definition of $k_n$ and Corollary \ref{degree_analysis_firstgen} allow us to conclude
\begin{equation}\label{g_n_9}
D_n \ge 0 \Rightarrow \begin{cases}
            \eta P(k,n) <  N(k,n) & \text{for }  1 \le k \le k_n-1\\
            N(k,n) \le \eta P(k,n) & \text{for } k_n \le k \le  K
                      \end{cases}
\end{equation}
\begin{equation}\label{g_n_10}
D_n < 0 \Rightarrow \begin{cases}
            \eta N(k,n) <  P(k,n) &\text{for } 1 \le k \le k_n-1\\
            P(k,n) \le \eta N(k,n) &\text{for } k_n \le k \le  K.
                      \end{cases}
\end{equation}

Translating the definition of the $\beta_B$ from \eqref{f_b_15} into the new notation, we see that
\begin{equation}\label{g_n_10_0}
 \beta_{i,k,n} = \begin{cases}
                 1 & \text{for } 1\le k < k_n,\; 1\le i \le M_{k,n} \\
                 \abs{D_n}^{1/2}\left(\sum\limits_{i=1}^{M_{k,n}}d_{i,k,n}^2 \right)^{-1/2}  &
               \text{for }           k_n \le k \le K,\; 1\le i \le M_{k,n}.
              \end{cases}
\end{equation}
This means that $G$ can be written
\begin{equation}
  G(x) = \sum_{n=1}^{N} \sum_{k=1}^K \sum_{i=1}^{M_{k,n}} \text{\large{$\chi$}}_{A_{i,k,n}} (x)
         \frac{d_{i,k,n} \beta_{i,k,n}}{\abs{x-c_{i,k,n}}} \tau_{i,k,n}(x),
\end{equation}
where $\tau_{i,k,n}$ is the unit tangent vector field in $A_{i,k,n}$.   In order to somewhat ease the notational burden, we define the following sets of indices.  The early and later generations are given respectively by
\begin{equation*}
 \begin{split}
 &S_e = \{ (n,k) \;|\; 1 \le n \le N, 1 \le k \le k_n -1  \} \\
 &S_l = \{ (n,k) \;|\; 1 \le n \le N, k_n \le k \le K \},
 \end{split}
\end{equation*}
and we similarly define the sets of early and later annuli by
\begin{equation*}
 \begin{split}
 &T_e = \{ (n,k,i) \;|\;  (n,k) \in S_e, 1\le i \le M_{k,n} \} \\
 &T_l = \{ (n,k,i) \;|\;   (n,k) \in S_l, 1\le i \le M_{k,n}  \}.
 \end{split}
\end{equation*}\\\\
Step 2.

In this step we will prove an intermediate bound on $t^2
\lambda_G(t)$.  We begin by breaking the distribution function for
$G$ up into two components determined by the value of $k_n$.
Indeed,
\begin{equation}\label{g_n_2}
\begin{split}
\lambda_G(t) &  = \sum_{n,k,i} \abs{A_{i,k,n}\cap \{\abs{G}>t\}} \\
       & = \sum_{T_e}\abs{A_{i,k,n}\cap\{\abs{G}>t\}} + \sum_{T_l} \abs{A_{i,k,n}\cap\{\abs{G}>t\}}\\
       & := A_1 +A_2.
\end{split}
\end{equation}
Applying Lemma \ref{annuli_exp_estimate} to $A_1$, we see that
\begin{equation}\label{g_n_3}
t^2 A_1 \le \pi \sum_{T_e} d_{i,k,n}^2 (1-e^{-2\tau_k}).
\end{equation}

To analyze the $A_2$ term we must take advantage of all of the notation created in the first step.  Particular attention must be paid to the generations after $k_n$ that come about as the result of mergings in which balls of nonzero degree are merged only with balls of zero degree.  These generations, which we call zero-merging generations, throw off a counting argument that we will use to bound the number of later generations (after $k_n$) in terms of the degrees of the balls in the $k_n^{th}$ generation.  Generations that are not zero-merging generations we call effective-merging generations.  The degrees of the annuli are not changed in a zero-merging generation, and the annuli of such a generation can be rearranged to fit concentrically outside the annuli of the previous generation.  Our strategy for dealing with zero-merging generations, then, is to collect successive zero-merging generations, group them with the preceding effective-merging generation, and utilize Lemma \ref{zero_merging_grouping} to regard the group as a single collection of annuli.

To this end, for each $n$ we define the sets
\begin{equation*}
\begin{split}
Z_n = \{  k \in \{k_n,\dotsc,K \} \;|\;& \text{each ball in } \mathcal{G}(t_k) \text{ contains at most one ball in}\\
& \mathcal{G}(t_k^-) \text{ of nonzero degree} \},
\end{split}
\end{equation*}
and
\begin{equation*}
I_n = \{k_n,\dotsc,K \} \backslash Z_n.
\end{equation*}
The generations in $Z_n$ are the zero-merging generations, and those in $I_n$ are the effective-merging generations.

Since $\abs{d_{i,k,n}}$ is a decreasing sequence with respect to $i$ for $k,n$ fixed, there must exist an integer $P_{k,n} \in \{1, \dotsc, M_{k,n} \}$ so that $d_{i,k,n} \neq 0$ for $i=1,\dotsc,P_{k,n}$ and $d_{i,k,n}=0$ for $i = P_{k,n} +1, \dotsc,M_{k,n}$.  Since the annuli of a zero-merging generation have the same degrees as the previous generation, we have that $P_{k,n} = P_{k-1,n}$.  We may assume, without loss of generality, that the ball ordering is such that $B_{i,k-1,n} \subset B_{i,k,n}$ and $d_{i,k,n} = d_{i,k-1,n}$ for $k\in Z_n$ and $i=1,\dotsc,P_{k,n}$.  To identify sequences of zero-merging generations that happen one after the other we write $Z_n = Z_n^1\cup \dotsb \cup Z_n^{m_n}$, where the $Z_n^j$ are maximal subsets of sequential integers, i.e. the integer connected components of $Z_n$.  All of the generations in $Z_n^j$ will be grouped with the generation preceding $Z_n^j$ and analyzed as a single entity with Lemma \ref{zero_merging_grouping}.  This preceding effective generation occurs at generation $l_n^j:=\min(Z_N^j)-1$.  We group it together with the generations in $Z_n^j$ by forming the collections $\tilde{Z}_n^j = Z_n^j \cup \{ l_n^j \}$.  Write the modified collection $\tilde{Z}_n = \tilde{Z}_n^1 \cup \dotsb \cup \tilde{Z}_n^{m_n}$, and $\tilde{I}_n = I_n \backslash \tilde{Z}_n$.  Note that $P_{k,n}$ is constant for $k\in \tilde{Z}_n^{j}$; we call this common value $P_n^j$.

We now split $A_2$ again:
\begin{equation}\label{g_n_4}
\begin{split}
A_2 &= \sum_{T_l} \abs{A_{i,k,n}\cap\{\abs{G}>t\}} \\
& = \sum_{n=1}^N\sum_{k\in \tilde{I}_n}\sum_{i=1}^{P_{k,n}} \abs{A_{i,k,n}\cap\{\abs{G}>t\}} + \sum_{n=1}^N \sum_{k\in      \tilde{Z}_n} \sum_{i=1}^{P_{k,n}} \abs{A_{i,k,n}\cap\{\abs{G}>t\}} \\
&:= B_1 + B_2.
\end{split}
\end{equation}
Applying Lemma \ref{annuli_exp_estimate} to $B_1$, we get
\begin{equation}\label{g_n_5}
t^2 B_1 \le \pi \sum_{n=1}^N \sum_{k\in \tilde{I}_n}\sum_{i=1}^{P_{k,n}}  (d_{i,k,n} \beta_{i,k,n})^2 (1-e^{-2\tau_k}) \le \pi \sum_{n=1}^N \sum_{k\in \tilde{I}_n}\sum_{i=1}^{P_{k,n}}  (d_{i,k,n} \beta_{i,k,n})^2.
\end{equation}
Upon inserting the values of $\beta_{i,k,n}$ from \eqref{g_n_10_0}, we find that
\begin{equation}\label{g_n_5_0}
 t^2 B_1 \le \pi \sum_{n=1}^N \sum_{k\in \tilde{I}_n} \abs{D_n} = \pi \sum_{n=1}^N \#(\tilde{I}_n) \abs{D_n},
\end{equation}
where $\#(\tilde{I}_n)$ denotes the cardinality of $\tilde{I}_n$.

To handle the $B_2$ term we note that
\begin{equation}
\begin{split}
 & \{ (n,k,i) \;|\; 1\le n\le N, k \in \tilde{Z}_n, 1\le i \le P_{k,n} \} \\ &=
 \bigcup_{\substack{1\le n \le N \\ 1 \le j \le m_n}} \{(n,k,i) \;|\; 1\le i \le P^j_n, k\in \tilde{Z}_n^j  \},
\end{split}
\end{equation}
and hence
\begin{equation}\label{g_n_6}
B_2 = \sum_{n=1}^N \sum_{j=1}^{m_n} \sum_{i=1}^{P_n^j} \sum_{k\in \tilde{Z}_n^j} \abs{A_{i,k,n}\cap\{\abs{G}>t\}}.
\end{equation}

When a zero-merging happens to a ball $B$ of nonzero degree, it is merged with a number of balls of zero degree.  The resulting ball has the same degree as $B$, and its radius is strictly larger than the radius of $B$.  Thus, we see that the radii hypothesis of Lemma \ref{zero_merging_grouping} is satisfied by $\{A_{i,k,n} \}$ for $k\in\tilde{Z}_n^j$, $i=1,\dotsc,P_{k,n}$.  Moreover, for $k\in\tilde{Z}_n^j$, we have that $d_{i,k,n} = d_{i,l_n^j,n}$ and $\beta_{i,k,n} = \beta_{i,l_n^j,n}$.  All hypotheses of Lemma \ref{zero_merging_grouping} are thus satisfied; applying it, for each $j,n$ we may bound
\begin{equation}\label{g_n_7}
t^2 \sum_{k\in \tilde{Z}_n^j} \abs{A_{i,k,n}\cap\{\abs{G}>t\}} \le  \pi (d_{i,l_n^j,n} \beta_{i,l_n^j,n})^2.
\end{equation}
Plugging in the values of $\beta_{i,k,n}$ from \eqref{g_n_10_0} then shows that
\begin{equation}\label{g_n_7_0}
 t^2 B_2 \le \sum_{n=1}^N \sum_{j=1}^{m_n} \pi \abs{D_n} = \sum_{n=1}^N \pi m_n \abs{D_n}.
\end{equation}

Recall that $I_n = \tilde{I}_n \cup \{l_n^1,\dotsc,l_n^{m_n}  \}$.  Hence $\#(I_n) = \#(\tilde{I}_n) + m_n$.  We then combine \eqref{g_n_4}, \eqref{g_n_5_0}, and \eqref{g_n_7_0} to get the estimate
\begin{equation}\label{g_n_8}
t^2 A_2 \le \pi \sum_{n=1}^N \#(I_n) \abs{D_n}.
\end{equation}
Together, \eqref{g_n_2}, \eqref{g_n_3}, and \eqref{g_n_8} prove that
\begin{equation}\label{g_n_1}
t^2 \lambda_G(t) \le  \pi \sum_{T_e} d_{i,k,n}^2 \left(
1-e^{-2\tau_k} \right) + \pi \sum_{n=1}^N \#(I_n) \abs{D_n},
\end{equation}
where $\#(I_n)$ is the cardinality of $I_n$.\\\\
Step 3.

In this step we will utilize the $\eta$ inequalities \eqref{g_n_9} and \eqref{g_n_10} to show that the energy excess, $F_{\varepsilon}(u,A) - \pi D \left( \log{\frac{r}{\varepsilon D}} -C\right),$ controls the first term on the right side of \eqref{g_n_1}.  To begin we modify an argument from the beginning of the proof of Theorem \ref{energy_bound}.   Define $V$ to be the union of the balls in $\mathcal{G}(s+\sigma)$.  Then, copying \eqref{e_b_3}, we can bound
\begin{equation}\label{g_n_11}
\begin{split}
 &F_{\varepsilon}^r(u,A,V) = \frac{1}{2} \int_{V} \rho^2 \abs{\nabla_A v}^2 + \frac{1}{2\varepsilon^2} (1-\rho^2)^2 + \abs{\nabla \rho}^2 + r^2(\curl{A})^2 \\
& \ge F_{\varepsilon}(\rho,V) + \int_0^{\frac{1}{2}} 2t \left( \frac{1}{2} \int_{V\backslash \omega_t} \abs{\nabla_A v}^2\right)dt + \frac{r^2}{8}\int_V (\curl{A})^2   \\
& + \int_{\frac{1}{2}}^{1-\delta} 2t \left( \frac{1}{2} \int_{V\backslash \omega_t} \abs{\nabla_A v}^2 + \frac{r^2}{2}\int_V (\curl{A})^2  \right)dt.
\end{split}
\end{equation}
For $t \in [0,1/2]$ the inclusions
\begin{equation}
 V \backslash \omega_t \supseteq V \backslash \omega_{1/2} \supseteq V \backslash \omega_{1/2}^{3/2}
\end{equation}
hold, and hence
\begin{equation}\label{g_n_11_0}
\begin{split}
  \int_0^{\frac{1}{2}} 2t \left( \frac{1}{2} \int_{V\backslash \omega_t} \abs{\nabla_A v}^2\right)dt
  &\ge \int_0^{\frac{1}{2}} 2t \left( \frac{1}{2} \int_{V\backslash \omega_{1/2}^{3/2}} \abs{\nabla_A v}^2 \right)dt \\
  & = \frac{1}{8} \int_{V\backslash \omega_{1/2}^{3/2}} \abs{\nabla_A v}^2.
\end{split}
\end{equation}
We now use \eqref{e_b_11} and \eqref{e_b_14} from Theorem \ref{energy_bound} to bound
\begin{equation}\label{g_n_12}
\begin{split}
&\int_{1/2}^{1-\delta} 2t \left( \frac{1}{2} \int_{V \backslash \omega_t} \abs{\nabla_A v}^2  + \frac{r^2}{2} \int_{V}(\curl{A})^2 \right)dt + \frac{3}{4} F_{\varepsilon}(\rho,V) \\
&\ge \pi D \left(\frac{3}{4} \log{\frac{r}{\varepsilon D}}-C  \right).
\end{split}
\end{equation}
Here we have used $D = \sum_{n=1}^N D_n$.  Assembling the bounds \eqref{g_n_11}, \eqref{g_n_11_0}, and \eqref{g_n_12} produces the bound
\begin{equation}\label{g_n_13}
\begin{split}
&F_{\varepsilon}^r(u,A,V) - \pi D \left(\log{\frac{r}{\varepsilon D}}-C \right) \\
& \ge \frac{1}{4}\left( F_{\varepsilon}(\rho,V) + \frac{1}{2} \int_{V \backslash \omega_{1/2}^{3/2}} \abs{\nabla_A v}^2 + \frac{r^2}{2} \int_V (\curl{A})^2  - \pi D \log{\frac{r}{\varepsilon D}}  \right).
\end{split}
\end{equation}

The argument in \eqref{e_b_8_0} shows that
\begin{equation}\label{g_n_14}
F_{\varepsilon}(\rho,V) - \pi D \left( \log{\frac{r}{\varepsilon D}}-C \right) \ge
\pi D \left( \log{\frac{r}{r_{\Omega_{\varepsilon}}(\omega_{1/2}^{3/2})}} -C \right).
\end{equation}
In order to use the logarithm terms they must be translated into the new notation.  Recalling \eqref{f_b_21} and changing the constant $C$ (larger but still universal), we see that
\begin{equation}\label{g_n_15}
\begin{split}
&\log{\frac{r}{r_{\Omega_{\varepsilon}}(\omega_{1/2}^{3/2})}} -C  = \log{\frac{r}{r(\mathcal{B}_0)}} + \log{\frac{3r(\mathcal{B}_0)}{16 r_{\Omega_{\varepsilon}}(\omega_{1/2}^{3/2})}} \\
&\le \log{\frac{r}{r(\mathcal{B}_0)}} + \log{\frac{r(\mathcal{C}(\sigma))}{r(\mathcal{C}_0)}}
 = \sum_{k=1}^{K} \tau_k.
\end{split}
\end{equation}
Combining \eqref{g_n_13} - \eqref{g_n_15} and again changing the constant, we arrive at
\begin{equation}\label{g_n_16}
\begin{split}
&F_{\varepsilon}^r(u,A,V) - \pi D \left(\log{\frac{r}{\varepsilon D}}-C \right) \\
& \ge \frac{1}{4}\left( \frac{1}{2} \int_{V \backslash \omega_{1/2}^{3/2}} \abs{\nabla_A v}^2 + \frac{r^2}{2}\int_V (\curl{A})^2 - \pi D \sum_{k=1}^K \tau_k   \right).
\end{split}
\end{equation}

We now translate the term on the right side of inequality \eqref{g_n_16} into the new notation and break it into two parts according to whether the generation is before or after generation $k_n$.  Indeed,
\begin{equation}\label{g_n_17}
\begin{split}
&\frac{1}{2} \int_{V \backslash \omega_{1/2}^{3/2}} \abs{\nabla_A
v}^2 + \frac{r^2}{2}\int_V (\curl{A})^2 -\pi D \sum_{k=1}^K \tau_k
\ge  \frac{1}{2} \sum_{T_e} \int_{A_{i,k,n}} \abs{\nabla_A v}^2 \ - \pi \sum_{S_e} \abs{D_n} \tau_k   \\
&+   \frac{1}{2} \sum_{T_l} \int_{A_{i,k,n}} \abs{\nabla_A v}^2 + r^2 (\curl{A})^2 - \pi \sum_{S_l}\abs{D_n} \tau_k
+ \sum_{n=1}^N \sum_{B \in \bar{B}_n \cap \mathcal{G}(t_{\bar{B}_n})} \frac{r^2}{2}\int_{B} (\curl{A})^2.
\end{split}
\end{equation}
For each $\bar{B}_n \in \mathcal{G}(s+\sigma)$ we consider $\bar{B}_n$ to have been grown from $\bar{B}_n \cap \mathcal{G}(t_{\bar{B}_n})$ and apply Corollary \ref{low_growth_old}; summing over $n$ gives
\begin{equation}
 \frac{1}{2} \sum_{T_l} \int_{A_{i,k,n}} \abs{\nabla_A v}^2 + r^2 (\curl{A})^2 \ge \pi D\left(\log{\frac{r}{r(\mathcal{G}(t_{\bar{B}_n}))}} - \log{2} \right).
\end{equation}
Note that if $t_{\bar{B}_n} \ge \sigma$, then
\begin{equation*}
\sum_{k=k_n}^K \tau_k = \log{\frac{r}{r(\mathcal{G}(t_{\bar{B}_n}))}},
\end{equation*}
whereas if $t_{\bar{B}_n}< \sigma$, then
\begin{equation*}
 \sum_{k=k_n}^K \tau_k = \log{\frac{r}{r(\mathcal{B}_0)}}  + \log{\frac{r(\mathcal{G}(\sigma))}{r(\mathcal{G}(t_{\bar{B}_n}))}} = \log{\frac{r}{r(\mathcal{G}(t_{\bar{B}_n}))}} + \log{\frac{3}{8}}
\end{equation*}
since $r(\mathcal{G}(\sigma)) = r(\mathcal{C}(\sigma)) = 3 r(\mathcal{B}_0)/8$ (see item 4 of Proposition \ref{init_balls}).
Then
\begin{equation}\label{g_n_18}
 \frac{1}{2} \sum_{T_l} \int_{A_{i,k,n}} \abs{\nabla_A v}^2 + r^2 (\curl{A})^2 - \pi \sum_{S_l}\abs{D_n} \tau_k  \ge -\pi C D,
\end{equation}
where $C$ is universal.

It remains to control the term corresponding to the early generations:
\begin{equation*}
Q:= \frac{1}{2} \sum_{T_e} \int_{A_{i,k,n}} \abs{\nabla_A v}^2 \ -
\pi \sum_{S_e} \abs{D_n} \tau_k    + \sum_{n=1}^N \sum_{B \in
\bar{B}_n \cap \mathcal{G}(t_{\bar{B}_n})} \frac{r^2}{2}\int_{B}
(\curl{A})^2.
\end{equation*}
We apply Lemma \ref{alt_low_growth_old} to each $B\in \bar{B}_n \cap \mathcal{G}(t_{\bar{B}_n})$ and sum to get
\begin{equation}\label{g_n_22}
 Q \ge \pi \sum_{S_e} \tau_k \left(\frac{2}{3} \sum_{i=1}^{M_{k,n}}  d_{i,k,n}^2 -\abs{D_n} \right).
\end{equation}

In order to control the difference in \eqref{g_n_22} we must now turn to the $\eta$ inequalities for generations before $k_n$.
If $D_n \ge 0$, $1\le k < k_n$, the inequality \eqref{g_n_9} allows us to estimate
\begin{equation}\label{g_n_23}
\begin{split}
  \sum_{i=1}^{M_{k,n}} d_{i,k,n}^2 & \ge \sum_{i=1}^{M_{k,n}} \abs{d_{i,k,n}}
              = \sum_{\substack{1\le i \le M_{k,n}\\d_{i,k,n}\ge0}} d_{i,k,n}
              + \sum_{\substack{1\le i \le M_{k,n}\\d_{i,k,n}<0}} \abs{d_{i,k,n}} \\
            & > (1+\eta) \sum_{\substack{1\le i \le M_{k,n}\\d_{i,k,n}\ge0}} d_{i,k,n}\\
        & \ge (1+\eta)D_n = (1+\eta)\abs{D_n}.
\end{split}
\end{equation}
If $D_n <0,$ we similarly get
\begin{equation*}
\sum_{i=1}^{M_{k,n}} d_{i,k,n}^2 > (1+\eta)\abs{D_n},
\end{equation*}
and so in either case we arrive at the estimate
\begin{equation}\label{g_n_24}
  -\abs{D_n} \ge -\frac{1}{1+\eta} \sum_{i=1}^{M_{k,n}} d_{i,k,n}^2.
\end{equation}
Putting \eqref{g_n_24} into \eqref{g_n_22} then shows that
\begin{equation}\label{g_n_25}
\begin{split}
Q &\ge \pi \frac{2\eta -1}{3(1+\eta)}\sum_{T_e}  \tau_k  d_{i,k,n}^2 \\
&\ge \pi \frac{2\eta -1}{6(1+\eta)}\sum_{T_e}  d_{i,k,n}^2 (1-e^{-2\tau_k}),
\end{split}
\end{equation}
where in the last inequality we have used the fact that
\begin{equation*}
 x \ge \frac{1}{2}(1-e^{-2x}) \text{ for } x\ge 0.
\end{equation*}
Finally, we use \eqref{g_n_16} -- \eqref{g_n_18} and \eqref{g_n_25} to conclude
\begin{equation}\label{g_n_26}
F_{\varepsilon}(u,A,V) - \pi D \left(\log{\frac{r}{\varepsilon D}}-C \right) \ge \pi \frac{2\eta -1}{24(1+\eta)}\sum_{T_e}  d_{i,k,n}^2 (1-e^{-2\tau_k}).
\end{equation}\\\\
Step 4.

In this step we use the $\eta$ inequalities to provide an upper bound for the second term on the right side of \eqref{g_n_1} by bounding $\#(I_n)$ in terms of $\abs{D_n}$ and $\eta$.  Fix $n$ and suppose that $k_n \le k \le K$.  For now take $D_n\ge 0$.  The inequality \eqref{g_n_9} allows us to bound
\begin{equation*}
\sum_{\substack{1\le i \le M_{k,n}\\d_{i,k,n}\ge0}} d_{i,k,n} = D_n + \sum_{\substack{1\le i \le M_{k,n}\\d_{i,k,n}< 0}} \abs{d_{i,k,n}} \le D_n + \eta \sum_{\substack{1\le i \le M_{k,n}\\d_{i,k,n}\ge0}} d_{i,k,n},
\end{equation*}
and so we can conclude that
\begin{equation}\label{g_n_27}
\sum_{\substack{1\le i \le M_{k,n}\\d_{i,k,n}\ge0}} d_{i,k,n} \le \frac{\abs{D_n}}{1-\eta}.
\end{equation}
We can use this estimate to bound $\#(I_n)$.  Each generation in $I_n$ is an effective-merging generation.  As such, the mergings of that generation include at least one ball of nonzero degree merging with another ball of nonzero degree, resulting in a decrease in the number of balls of nonzero degree.  So, the number of effective generations, $\#(I_n)$, is bounded by the number of nonzero degree balls in the $k_n$ generation.  This quantity can then be bounded in terms of $D_n$ and $\eta$.  Indeed,
\begin{equation}\label{g_n_28}
\begin{split}
  \#(I_n) & \le \# \text{ of nonzero degree balls in generation }k_n \\
          & \le \sum_{i=1}^{M_{k_n,n}}\abs{d_{i,k_n,n}}
            = \sum_{\substack{1\le i \le M_{k_n,n}\\d_{i,k_n,n}\ge0}}\abs{d_{i,k_n,n}}
            + \sum_{\substack{1\le i \le M_{k_n,n}\\d_{i,k_n,n}<0}}\abs{d_{i,k_n,n}}\\
          & \le (1+\eta)\sum_{\substack{1\le i \le M_{k_n,n}\\d_{i,k_n,n}\ge0}}d_{i,k_n,n} \\
      & \le \frac{1+\eta}{1-\eta}\abs{D_n}.
\end{split}
\end{equation}
If $D_n < 0$ then \eqref{g_n_10} and a similar argument show that \eqref{g_n_28} still holds.  Hence
\begin{equation}\label{g_n_29}
\pi \sum_{n=1}^N \#(I_n) \abs{D_n} \le \pi \frac{1+\eta}{1-\eta} \sum_{n=1}^N \abs{D_n}^2.
\end{equation}\\\\
Step 5.

We now conclude the proof by combining \eqref{g_n_1}, \eqref{g_n_26}, and  \eqref{g_n_29} to get the inequality
\begin{equation}
t^2 \lambda_G(t) \le \pi \frac{1+\eta}{1-\eta} \sum_{n=1}^N
\abs{D_n}^2 + \frac{24(1+\eta)}{2\eta -1}\left(
F_{\varepsilon}^r(u,A,V) - \pi D \left(\log{\frac{r}{\varepsilon
D}}-C \right)  \right).
\end{equation}
Using Lemma \ref{l2_inf_normability} and switching back to our original notation then proves \eqref{g_n_30}.
\end{proof}

\section{Jerrard's construction}
In the above results we have modified and improved
 the vortex ball
construction of Sandier, introduced in \cite{sa}, and presented in
an updated form in \cite{ss_book}.  The purpose of this section is
to show that the methods of this paper can be applied equally well
to the other version of the vortex ball construction, developed by
Jerrard in \cite{j}.  The two constructions are not at all
dissimilar, so it is no surprise that the above methods still
work.  For completeness, though, we highlight the differences in
the two constructions and outline the modifications necessary to
make the above ideas work with  Jerrard's construction.  In the
interest of brevity we discuss only the case without magnetic
field.

There are three main differences between the ball construction
employed above and that of \cite{j}.  The Jerrard construction
grows finite collections of disjoint balls from an initial small
collection to a final large collection, employing mergings when
grown balls become tangent.  However, a collection of disjoint
balls $\{ B_i \}$ is not grown uniformly, as we grow them above,
but instead according to the parameter
\begin{equation*}
 s = \min_{i} \frac{r_i}{\abs{d_i}},
\end{equation*}
where $d_i = \deg(u,\partial B_i)$ and $r_i$ is the radius of $B_i$.  There is no guarantee that this parameter is uniform throughout the collection (hence the minimum in the definition of $s$), and as a result, only balls for which the minimum $s$ is achieved are grown.  Note that as a ball is grown without merging, its degree does not vary, so increasing $s$ amounts to increasing the radius of the ball.  Moreover, for the subcollection of balls in $\{B_i\}$ that achieve $s$, if we write $s^{new}$ for the increased parameter and $r_i^{new}$ for the increased radii, we see that
\begin{equation*}
 \frac{s^{new}}{s} = \frac{r_i^{new}}{d_i}\frac{d_i}{r_i} = \frac{r_i^{new}}{r_i},
\end{equation*}
and so all of the annuli formed by deleting the old balls from the new ones have the same conformal type.  The use of this parameter causes trouble above since $r(\mathcal{B}(t)) \neq e^t r(\mathcal{B}_0)$.

The second major difference in the two methods is in how they pass from lower bounds on circles, which in both methods are most conveniently calculated by estimating $\frac{1}{2}\int_{\partial B(a,r)} \abs{\nabla v}^2$
from below, to lower bounds of $\frac{1}{2}\int\abs{\nabla u}^2$ on annuli and balls.  Above we employ the co-area formula in Lemma \ref{co_area_subset} and in \eqref{e_b_3} of Theorem \ref{energy_bound} to accomplish this.  The Jerrard method writes $u = \rho v$, with $\rho = \abs{u}$, and expands the energy as
\begin{equation*}
 \frac{1}{2} \int_{\partial B(a,r)} \abs{\nabla u }^2 + \frac{1}{2\varepsilon^2}(1-\abs{u}^2)^2 = \frac{1}{2} \int_{\partial B(a,r)} \abs{\nabla \rho}^2 + \frac{1}{2\varepsilon^2}(1-\rho^2)^2 + \frac{1}{2} \int_{\partial B(a,r)} \rho^2 \abs{\nabla v}^2.
\end{equation*}
Lemmas 2.4 and 2.5 of \cite{j} then show that
\begin{equation*}
 \frac{1}{2} \int_{\partial B(a,r)} \rho^2 \abs{\nabla v}^2 \ge \pi \frac{m^2 d^2}{r},
\end{equation*}
 and
\begin{equation*}
 \frac{1}{2} \int_{\partial B(a,r)} \abs{\nabla \rho}^2 + \frac{1}{2\varepsilon^2}(1-\rho^2)^2 \ge \frac{1}{c\varepsilon} (1-m)^2,
\end{equation*}
where $c$ is a universal constant and  $m = \min\{1,\inf\limits_{\partial B(a,r)} \rho\}$.  These two bounds are combined with the energy expansion to find
\begin{equation*}
 \frac{1}{2} \int_{\partial B(a,r)} \abs{\nabla u }^2 + \frac{1}{2\varepsilon^2}(1-\abs{u}^2)^2 \ge \inf_{m \in[0,1]} \left( \pi \frac{m^2 d^2}{r}+ \frac{1}{c\varepsilon} (1-m)^2 \right) =: \lambda_\varepsilon(r,d).
\end{equation*}
One readily verifies that $\lambda_\varepsilon(r,d) \ge \lambda_\varepsilon(r/\abs{d},1)$ and that
\begin{equation}\label{jerrard_0}
\lambda_\varepsilon(r,1) = \frac{\pi}{r + c\varepsilon \pi}.
\end{equation}
The function $\Lambda_\varepsilon(s) = \int_0^s \lambda_\varepsilon(r,1) dr = \pi \log(1+ \frac{s}{c\varepsilon \pi})$ is then introduced, and lower bounds on annuli are calculated by integrating on circles:
\begin{equation*}
\begin{split}
 \frac{1}{2} \int_{B(a,r_1) \backslash B(a,r_0)} \abs{\nabla u }^2 + \frac{1}{2\varepsilon^2}(1-\abs{u}^2)^2 &\ge
 \int_{r_0}^{r_1} \lambda_\varepsilon(r,d) dr \ge \abs{d} \int_{r_0/\abs{d}}^{r_1/\abs{d}} \lambda_\varepsilon(r,1)dr\\ &=\abs{d}(\Lambda_\varepsilon(r_1/\abs{d}) - \Lambda_\varepsilon(r_0/\abs{d})).
\end{split}
\end{equation*}
Note that this bound justifies the use of $s = r/d$ as the growth parameter.

The third major difference is in the nature of the lower bounds.  The method above produces lower bounds on the total collection of balls but can not say much about the energy content of any given ball in the collection.  Because of its use of the $\Lambda_\varepsilon$ function, which only depends on the parameter $s$, the Jerrard construction can localize the lower bounds to each ball in the collection.  In particular, Proposition 4.1 of \cite{j}, the analogue of our Theorem \ref{energy_bound}, shows that there exists a $\sigma_0$ such that for any $0 \le \sigma \le \sigma_0$
there exists a collection of disjoint balls $\{B_i\}$ with radii $r_i$ and degrees $d_i$ such that
\begin{equation*}
 \frac{1}{2} \int_{B_i \cap \Omega} \abs{\nabla u }^2 + \frac{1}{2\varepsilon^2}(1-\abs{u}^2)^2 \ge \frac{r_i}{s} \Lambda_\varepsilon(s),
\end{equation*}
where $s = \min\limits_{i} (r_i/\abs{d_i}) \in [\sigma/2,\sigma]$.  In particular this implies that
\begin{equation*}
 \frac{1}{2} \int_{B_i \cap \Omega} \abs{\nabla u }^2 + \frac{1}{2\varepsilon^2}(1-\abs{u}^2)^2 \ge \pi \abs{d_i} \log\left(1+ \frac{\sigma}{2c\pi \varepsilon} \right).
\end{equation*}
The proof of this result follows from a line of reasoning similar to what led to Theorem \ref{energy_bound}.  An initial collection of balls $\{B_i\}$ with radii $r_i \ge \varepsilon$ is found (Proposition 3.3 of \cite{j}) that covers $\{ \abs{u} \le 1/2\}$ and on which
\begin{equation}\label{jerrard_1}
 \frac{1}{2} \int_{B_i \cap \Omega} \abs{\nabla u }^2 + \frac{1}{2\varepsilon^2}(1-\abs{u}^2)^2 \ge c_0 \frac{r_i}{\varepsilon} \ge \frac{r_i}{s} \Lambda_\varepsilon(s),
\end{equation}
where $c_0$ is a universal constant.  These balls are then grown into the final balls according to the ball growth lemma, but used with the parameter $s$ as the growth parameter.  It is then shown that growth and merging preserves the form of the lower bound \eqref{jerrard_1}, i.e. that if the bound holds with one value of $s$, it also holds with the value of $s$ obtained after growing the balls. \medskip

In order to utilize our completion of the square trick to extract the new term we must only present a modification of Lemma \ref{low_boundary} designed to work with the minimization of $m$ trick.  The rest of the argument follows from simple modifications of the arguments in \cite{j} that we will only sketch.

\begin{lem}\label{jerrard_mod}
 Let $B = B(a,r)$ and suppose that $u: \partial B \rightarrow \mathbb{C}$ is $C^1$ and that $\abs{u}>c\ge0$ on $\partial B$.  Write $u = \rho v$ with $\rho = \abs{u}$, and define the function
 \begin{equation}
  G = \frac{d m^2 \beta}{\rho^2 r}\tau,
 \end{equation}
where $d = \deg(u,\partial B)$, $m = \min\{1,\inf\limits_{\partial B(a,r)} \rho\}$, $\tau$ is the oriented unit tangent vector field to $\partial B$, and $\beta \in [0,1]$ is a constant.  Then
\begin{equation}\label{j_mod_0}
 \frac{1}{2} \int_{\partial B} \rho^2 \abs{\nabla v}^2 \ge \frac{1}{2} \int_{\partial B} \rho^2 \abs{\nabla v- G}^2 + \pi \frac{d^2 m^2\beta}{r}.
\end{equation}
\end{lem}
\begin{proof}
 Arguing as in Lemma \ref{low_boundary}, we find that
 \begin{equation}
  \frac{1}{2} \int_{\partial B} \rho^2 \abs{\nabla v}^2 = \frac{1}{2} \int_{\partial B} \rho^2 \abs{\nabla v- G}^2 + 2\pi d \frac{dm^2 \beta}{r} - \frac{d^2 m^4 \beta^2}{2r^2}\int_{\partial B} \frac{1}{\rho^2}.
 \end{equation}
Then the definition of $m$ implies that
\begin{equation}
 2\pi d \frac{dm^2 \beta}{r} - \frac{d^2 m^4 \beta^2}{2r^2}\int_{\partial B} \frac{1}{\rho^2} \ge \pi \frac{d^2 m^2}{r}(2\beta - \beta^2) \ge \pi \frac{d^2 m^2 \beta}{r},
\end{equation}
where the last inequality follows from the fact that $0 \le \beta \le 1$.  This proves the result.

\end{proof}

This result may be used in conjunction with Lemma 2.5 of \cite{j}, borrowing half of that energy to absorb into the novel term, to arrive at the lower bound
\begin{equation}
 \frac{1}{2} \int_{\partial B} \abs{\nabla u }^2 + \frac{1}{2\varepsilon^2}(1-\abs{u}^2)^2 \ge
 \frac{1}{4} \int_{\partial B} \abs{\nabla u - iuG}^2 +  \inf_{m \in[0,1]} \left( \pi \frac{m^2 d^2\beta}{r}+ \frac{1}{c\varepsilon} (1-m)^2 \right) .
\end{equation}
In order to gain the ability to localize the estimates in each ball, we must have that $\lambda_\varepsilon(r,d)$ is independent of $\beta$ and that the homogeneity inequality $\lambda_\varepsilon(r,d) \ge \lambda_\varepsilon(r/\abs{d},1)$ holds.  The first of these requires us to set $\beta =1$ in the above, which precludes the special choice of $\beta$ needed to make Proposition \ref{g_norm} work.  The second requires us to throw away the $d^2$ terms in favor of $\abs{d}$.  So, there is a tradeoff: the price we pay for localizing the estimates is a loss of control of the $L^{2,\infty}$ norm of the auxiliary function $G$.  This choice leads to the lower bound on circles
\begin{equation}
 \frac{1}{2} \int_{\partial B} \abs{\nabla u }^2 + \frac{1}{2\varepsilon^2}(1-\abs{u}^2)^2 \ge
 \frac{1}{4} \int_{\partial B} \abs{\nabla u - iuG}^2 + \lambda_\varepsilon(r/\abs{d},1),
\end{equation}
where $\lambda_\varepsilon$ is as defined in \eqref{jerrard_0}, but with the universal constant doubled, and $G= \frac{d m^2}{\rho^2 r}\tau$.  The bound on circles leads to bounds on annuli by integrating; indeed,
\begin{equation}
 \begin{split}
  \frac{1}{2} \int_{B(a,r_1) \backslash B(a,r_0)} \abs{\nabla u }^2 + \frac{1}{2\varepsilon^2}(1-\abs{u}^2)^2 &\ge
  \frac{1}{4} \int_{B(a,r_1) \backslash B(a,r_0)} \abs{\nabla u - iuG}^2
 \\ & +\abs{d}(\Lambda_\varepsilon(r_1/\abs{d}) - \Lambda_\varepsilon(r_0/\abs{d})),
\end{split}
\end{equation}
where now we take $G(x)= \frac{d m^2}{\rho(x)^2 \abs{x-a}}\tau(x)$.

Now, to achieve a bound of the form \eqref{jerrard_1} but with the $L^2$ difference with $iuG$ included, we use Lemma \ref{jerrard_mod} in the Jerrard construction.  As above, we define the function $G$ to vanish in the initial collection of balls obtained in Proposition 3.3 of \cite{j}.  Then we trivially modify \eqref{jerrard_1} to read (since $G=0$ there)
\begin{equation}
\begin{split}
  \frac{1}{2} \int_{B_i \cap \Omega} \abs{\nabla u }^2 + \frac{1}{2\varepsilon^2}(1-\abs{u}^2)^2 & \ge  \frac{c_0 r_i}{2 \varepsilon}  + \frac{1}{4} \int_{B_i \cap \Omega} \abs{\nabla u}^2 \\
  & \ge \frac{r_i}{s} \Lambda_\varepsilon(s)+ \frac{1}{4} \int_{B_i \cap \Omega} \abs{\nabla u - iuG}^2.
\end{split}
\end{equation}
We then take $G$ to vanish in all of the non-annular regions of the balls constructed in Proposition 4.1 of \cite{j}.  The estimates in these balls, like the original Sandier estimates, discard the energy of the non-annular regions.  We retain it and rewrite it as a $\int \abs{\nabla u - iuG}^2$ term, which is possible since $G=0$ there.  Then, adding in the extra $G$ term in the annular regions, we arrive at the modification.

\begin{prop}
There exists a $\sigma_0$ such that for any $0 \le \sigma \le \sigma_0$
there exists a collection of disjoint balls $\{B_i\}$ with radii $r_i$ and degrees $d_i$ such that
\begin{equation*}
 \frac{1}{2} \int_{B_i \cap \Omega} \abs{\nabla u }^2 + \frac{1}{2\varepsilon^2}(1-\abs{u}^2)^2 \ge \frac{1}{4} \int_{B_i \cap \Omega} \abs{\nabla u - iuG}^2 + \frac{r_i}{s} \Lambda_\varepsilon(s),
\end{equation*}
where $s = \min\limits_{i} (r_i/\abs{d_i}) \in [\sigma/2,\sigma]$.  In particular this implies that
\begin{equation*}
 \frac{1}{2} \int_{B_i \cap \Omega} \abs{\nabla u }^2 + \frac{1}{2\varepsilon^2}(1-\abs{u}^2)^2 \ge \frac{1}{4} \int_{B_i \cap \Omega} \abs{\nabla u - iuG}^2+\pi \abs{d_i} \log\left(1+ \frac{\sigma}{2c\pi \varepsilon} \right).
\end{equation*}
\end{prop}

\newpage


\end{document}